\documentclass[12pt]{article}
\usepackage[]{amsmath,amssymb}
\usepackage{amscd}
\usepackage{latexsym}
\usepackage{cite}
\usepackage{amsthm}

\newtheorem{definition}{Definition}[section]
\newtheorem{theorem}[definition]{Theorem}
\newtheorem{lemma}[definition]{Lemma}
\newtheorem{corollary}[definition]{Corollary}

\newtheorem{conjecture}[definition]{Conjecture}
\newtheorem{problem}[definition]{Problem}
\newtheorem{note}[definition]{Note}

\newtheorem{proposition}[definition]{Proposition}

\def\Z{\mathbb Z}
\def\F{\mathbb F}

\begin{document}
\title{\bf  
The $q$-Onsager algebra and the \\
positive part of $U_q({\widehat{\mathfrak{sl}}}_2)$
}
%Tatsuro Ito\footnote{Supported in part by JSPS grant
%18340022.} $\;$   and
\author{
Paul Terwilliger}
\date{}
%\footnote{This author gratefully acknowledges 
%support from the FY2007 JSPS Invitation Fellowship Program
%for Reseach in Japan (Long-Term), grant L-07512.}
%}
%\date{}
%to get date printout, comment out above line

\maketitle
\begin{abstract}
The positive part $U^+_q$ of 
$U_q({\widehat{\mathfrak{sl}}}_2)$ has a presentation 
by two generators $X,Y$ that satisfy the $q$-Serre relations.
The $q$-Onsager algebra $\mathcal O_q$  has a presentation
by two generators
$A,B$ that satisfy the $q$-Dolan/Grady relations.
We give two results that describe how $U^+_q$ and $\mathcal O_q$ are related.
First,
we consider the filtration of $\mathcal O_q$ whose $n$th component
is spanned by the products of at most $n$ generators.
We show that the associated graded algebra 
 is isomorphic to $U^+_q$.
Second, we introduce an algebra
$\square_q$ and show how it is related to both $U^+_q$ and
$\mathcal O_q$. The algebra $\square_q$ is defined by
generators and relations. The generators are
$\lbrace x_i \rbrace_{i \in
\mathbb Z_4}$ where $\mathbb Z_4$ is the cyclic group of order 4.
For $i \in \mathbb Z_4$ the generators $x_i, x_{i+1}$ satisfy a $q$-Weyl relation,
and $x_i,x_{i+2}$ satisfy the $q$-Serre relations.
We show that $\square_q$ is related to
$U^+_q$  in the following way.
Let
$ \square^{\rm even}_q$
(resp. 
$ \square^{\rm odd}_q$)
denote the subalgebra
of $ \square_q$ generated by
 $x_0, x_2$ (resp. 
 $x_1, x_3$).
We show that 
(i) there exists an algebra isomorphism
$U^+_q \to  \square^{\rm even}_q$ that sends
$X\mapsto x_0$ and
$Y\mapsto x_2$;
(ii) there exists an algebra isomorphism
$U^+_q \to  \square^{\rm odd}_q$ that sends
$X\mapsto x_1$ and
$Y\mapsto x_3$;
(iii) the multiplication map
 $\square^{\rm even}_q
\otimes
 \square^{\rm odd}_q
  \to   \square_q$,
 $u \otimes v  \mapsto  uv$ is an isomorphism of
 vector spaces.
We show that $\square_q$ is related to
 $\mathcal O_q$ in the following way.
For nonzero scalars $a,b$
there exists an injective algebra 
homomorphism
$ \mathcal O_q \to \square_q$ that sends
$A \mapsto a x_0+ a^{-1} x_1$ and
$B \mapsto
b x_2+ b^{-1} x_3$.

 \end{abstract}
\section{Introduction}
There is a family of algebras called tridiagonal algebras
\cite[Definition~3.9]{TwoRel}
that come up in the theory of $Q$-polynomial distance-regular
graphs 
\cite[Lemma~5.4]{tersub3}
 and tridiagonal pairs
\cite[Theorem~10.1]{TD00},
%\cite[Proposition~8.5]{qRacahIT},
%\cite{ITaug},
%\cite[Theorem~1.12]{LS99},
\cite[Theorem~3.10]{TwoRel}.
A tridiagonal algebra has a presentation by
two generators
and two relations of a certain kind, called  tridiagonal relations
\cite[Definition~3.9]{TwoRel}.
One example of a tridiagonal algebra is 
the positive part 
$U^+_q$ of
the quantum affine algebra 
$U_q(\widehat {\mathfrak {sl}}_2)$
\cite[Corollary~3.2.6]{lusztig},
\cite[Lines~(18),~(19)]{LS99}.
%%\cite{cp3},
%%\cite{shape}, 
%%\cite{nonnil},
%%\cite{qtet}. 
The algebra $U^+_q$ has a presentation by generators $X,Y$ and
relations
\begin{eqnarray}
&&
X^3 Y - \lbrack 3 \rbrack_q X^2 Y X +
\lbrack 3 \rbrack_q X Y X^2 -Y X^3 = 0,
\label{eq:XXXYIntro}
\\
&&
Y^3 X - \lbrack 3 \rbrack_q Y^2 X Y +
\lbrack 3 \rbrack_q Y X Y^2 -X Y^3 = 0,
\label{eq:YYYXIntro}
\end{eqnarray}
where $\lbrack 3 \rbrack_q = (q^3-q^{-3})/(q-q^{-1})$.
The relations (\ref{eq:XXXYIntro}),
(\ref{eq:YYYXIntro}) are called
the
$q$-Serre relations
\cite{lusztig}.
%%A finite-dimensional irreducible 
%%$U^+_q$-module on which $X,Y$ are not nilpotent is
%%essentially the same thing as a tridiagonal pair 
%%of $q$-geometric type
%%\cite[Lemma~4.8]{TD00},
%%\cite[Theorem~2.7]{nonnil}.
Applications of $U^+_q$ to tridiagonal pairs can be found in
\cite[Example~1.7]{TD00},
\cite{shape},
\cite[Section~2]{nonnil},
\cite[Lemma~4.8]{TwoRel}.
The algebra 
$U^+_q$ plays a prominant role in the theory of
$U_q(\widehat {\mathfrak {sl}}_2)$
\cite{beck3,
beck2,
beck,
damiani,
uqsl2hat,
qtet,
lusztig}.
Another example of a tridiagonal algebra is 
the $q$-Onsager algebra $\mathcal O_q$
\cite[Section 2]{bas2},
\cite[Section 1]{bas1}.
 This algebra
has a presentation by generators $A,B$
and relations
\begin{eqnarray}
&&
A^3 B - \lbrack 3 \rbrack_q A^2 B A +
\lbrack 3 \rbrack_q A B A^2 -B A^3 = (q^2-q^{-2})^2 (BA-AB),
\label{eq:AAABIntro}
\\
&&
B^3 A - \lbrack 3 \rbrack_q B^2 A B +
\lbrack 3 \rbrack_q B A B^2 -A B^3 = (q^2-q^{-2})^2 (AB-BA).
\label{eq:BBBAIntro}
\end{eqnarray}
The relations
(\ref{eq:AAABIntro}),
(\ref{eq:BBBAIntro})
are called the $q$-Dolan/Grady relations
\cite[Line~(5)]{bas1}.
Applications of 
$\mathcal O_q$ to tridiagonal pairs
can be found in
\cite{TD00,
qRacahIT,
ITaug,
LS99,
TwoRel}.
The algebra $\mathcal O_q$  has applications to quantum integrable models
\cite{bas2,             
bas1,              
BK05,           
  bas4,          
bas5,           
    bas6,      
   bas7,  
bas8,
 basnc,   
basXXZ,
basKojima},
reflection equation algebras
 \cite{basnc},
and
coideal subalgebras 
\cite{bc, kolb}. 
There are algebra homomorphisms
from $\mathcal O_q$ into
$U_q(\widehat {\mathfrak {sl}}_2)$ 
 \cite[line (3.15)]{basXXZ}, 
 \cite[line (3.18)]{basXXZ},
 \cite[Example~7.6]{kolb},
the $q$-deformed loop algebra
$U_q(L({\mathfrak{sl}}_2))$
\cite[Prop.~2.2]{bas6},
\cite[Prop.~8.5]{qRacahIT},
\cite[Props.~1.1,~1.13]{ITaug},
and the 
 universal Askey-Wilson algebra $\Delta_q$ 
\cite[Sections~9,~10]{uaw}.
\medskip 

\noindent Consider how
$U^+_q$
and
$\mathcal O_q$
are related.
These algebras have at least a superficial
resemblance, since
for
the
$q$-Serre relations
and
$q$-Dolan/Grady relations
their left-hand sides match.
In this paper our goal is to describe how
$U^+_q$ and
$\mathcal O_q$  are related on an
algebraic level.
Our first main result is 
summarized as follows. 
For notational convenience abbreviate $\mathcal O = \mathcal O_q$.
We consider the filtration
$\mathcal O_0 \subseteq \mathcal O_1 \subseteq \mathcal O_2 \subseteq \ldots $
of $\mathcal O$ such that for $n \in \mathbb N$
the subspace
$\mathcal O_n$ is spanned by the products
$g_1 g_2 \cdots g_r$ for which $0 \leq r \leq n$
and $g_i$ is among $A,B$ for $1 \leq i \leq r$.
We consider the associated graded algebra
$\overline{\mathcal O}$
in the sense of
\cite[p.~203]{carter}.
We show that 
the algebras
$U^+_q$
and
$\overline{\mathcal O}$
are isomorphic.
For our second main result, we introduce 
an algebra $\square_q$ and show how it is related
to both
$U^+_q$ and
$\mathcal O_q$. Let
$\mathbb Z_4 = \mathbb Z/4 \mathbb Z$ denote the cyclic
group of order 4.
The algebra $\square_q$ has a presentation by
generators $\lbrace x_i \rbrace_{i \in \mathbb Z_4}$ and relations
\begin{eqnarray*}
&&
\quad \qquad \qquad 
\frac{q x_i x_{i+1}-q^{-1}x_{i+1}x_i}{q-q^{-1}} = 1,
\\
&&
x^3_i x_{i+2} - \lbrack 3 \rbrack_q x^2_i x_{i+2} x_i +
\lbrack 3 \rbrack_q x_i x_{i+2} x^2_i -x_{i+2} x^3_i = 0.
\end{eqnarray*}
We show that $\square_q$ is related to
$U^+_q$  in the following way.
Let
$ \square^{\rm even}_q$
(resp. 
$ \square^{\rm odd}_q$)
denote the subalgebra
of $ \square_q$ generated by
 $x_0, x_2$ (resp. 
 $x_1, x_3$).
We show that 
(i) there exists an algebra isomorphism
$U^+_q \to  \square^{\rm even}_q$ that sends
$X\mapsto x_0$ and
$Y\mapsto x_2$;
(ii) there exists an algebra isomorphism
$U^+_q \to  \square^{\rm odd}_q$ that sends
$X\mapsto x_1$ and
$Y\mapsto x_3$;
(iii) the multiplication map
 $\square^{\rm even}_q
\otimes
 \square^{\rm odd}_q
  \to   \square_q$,
 $u \otimes v  \mapsto  uv$ is an isomorphism of
 vector spaces.
We show that $\square_q$ is related to
 $\mathcal O_q$ in the following way.
For nonzero scalars $a,b$
there exists an injective algebra 
homomorphism
$ \mathcal O_q \to \square_q$ that sends
\begin{eqnarray}
A \mapsto a x_0+ a^{-1} x_1, \qquad \qquad B \mapsto
b x_2+ b^{-1} x_3.
\label{eq:IntroMap}
\end{eqnarray}
Our two main results are obtained under the following
assumptions.
The underlying field is arbitrary.
The scalar $q$ is nonzero
 and  $q^2 \not=1$.
Our two main results are closely related, and
will be proved more or less simultaneously.
These proofs use only linear algebra, and
do not employ facts invoked from the literature.
Our proof strategy is to introduce several
algebras $\widetilde \square_q$, $\widehat \square_q$
that are related to $\square_q$ via surjective algebra homomorphisms
$\widetilde \square_q \to \widehat \square_q \to \square_q$.
The algebra $\widetilde \square_q$ is very general, and
an explicit basis will be given. 
%%%%%%%%%%%%%%%%%%%%%%%%%%%
%Using this  basis
%we will obtain some facts about
%$\widetilde \square_q$, which are transported to
%$\widehat \square_q$ and then 
%$\square_q$ 
%via the
%homomorphisms
%$\widetilde \square_q \to \widehat \square_q \to \square_q$.
%The resulting facts about
%$\widehat \square_q, \square_q$  will yield
%an algebra homomorphism $\mathcal O_q \to \widehat \square_q$
% such that 
%%%%%%%%%%%%%%%%%%%%%%%%%
Using this  basis
we will obtain some facts about
$\widetilde \square_q$, which give facts about
$\widehat \square_q$ and
$\square_q$ 
via the
homomorphisms
$\widetilde \square_q \to \widehat \square_q \to \square_q$.
These facts
 yield
an algebra homomorphism $\mathcal O_q \to \widehat \square_q$
 such that 
the composition
$\mathcal O_q \to \widehat \square_q \to \square_q$  is
injective. This composition is the homomorphism
(\ref{eq:IntroMap}).
\medskip

\noindent   Near the end of the paper we will discuss how
$\widehat \square_q$ and $\square_q$ are related to
$U_q(\widehat {\mathfrak {sl}}_2)$ and the $q$-tetrahedrom
algebra $\boxtimes_q$ from
\cite{qtet}.  We will obtain a commuting diagram of
algebra homomorphisms
\begin{equation*}
\begin{CD}
\widehat \square_q @>>>
                      U_q(\widehat {\mathfrak {sl}}_2) 
	   \\ 
          @VVV                   @VVV \\
            \square_q @>>> 
             \boxtimes_q 
                   \end{CD}
\end{equation*}
In this diagram the homomorphism
$\square_q \to \boxtimes_q$ is injective.
Using the diagram we will explain
the homomorphisms $\mathcal O_q \to
                      U_q(\widehat {\mathfrak {sl}}_2) $
		      and 
 $\mathcal O_q \to
                      U_q(L({\mathfrak {sl}}_2)) $
		      that we mentioned earlier in this
                     section. 

%%%%%%%%%%%%%%%%%%%%%%%%%%%%%%%%%%%%%%%%%%%%%%%%

\section{Preliminaries}

\noindent We now begin our formal argument. 
Throughout this paper the following notation
and assumptions are
in effect.
Recall the
 natural numbers 
$\mathbb N = \lbrace 0, 1, 2,\ldots \rbrace$ and
integers
$\mathbb Z = \lbrace 0, \pm 1, \pm 2,\ldots \rbrace$.
We will be discussing algebras.
An algebra is meant to be associative
and have a 1. A subalgebra has the same 1 as the parent
algebra. Let $\mathbb F$ denote a field and
let $\mathcal A$ denote an $\mathbb F$-algebra.
Let $H,K$ denote subspaces of
the $\mathbb F$-vector space $\mathcal A$.
 Then $HK$
denotes the subspace  
of $\mathcal A$ spanned by
 $\lbrace hk | h\in H, k \in K\rbrace$.
By an {\it $\mathbb N$-grading} of $\mathcal A$ 
we mean a sequence $\lbrace \mathcal A_n\rbrace_{n \in \mathbb N}$ 
such that (i) each $\mathcal A_n$ is a subspace 
of the $\mathbb F$-vector space $\mathcal A$;
(ii) $1 \in \mathcal A_0$;
(iii) the sum $\mathcal A = \sum_{n \in \mathbb N} \mathcal A_n$ is direct;
(iv) $\mathcal A_r \mathcal A_s \subseteq A_{r+s}$ for 
$r,s\in \mathbb N$.
A $\mathbb Z$-grading of $\mathcal A$ is similarly defined.
\medskip

\noindent 
We will be discussing algebras defined by generators and relations.
Let $T$ denote the $\mathbb F$-algebra with generators $x,y$ and
no relations; $T$ is often called a free algebra or tensor algebra.
The generators $x,y$ will be called {\it standard}.
For $n \in \mathbb N$, a {\it word of length $n$} in
$T$ is a product $g_1 g_2 \cdots g_n$ such that
$g_i$ is a standard generator 
for $1 \leq i \leq n$.
We interpret the word of length zero to be the multiplicative
identity in $T$. The words in $T$ form a basis for the
$\mathbb F$-vector space $T$.
For $n \in \mathbb N$ 
the words of length $n$  in $T$ form a basis for
a subspace $T_n$ of $T$.
For example,
$1$ is a basis for
$T_0$ and $x,y$ is a basis for $T_1$.
By construction 
the sum $T=\sum_{n\in \mathbb N} T_n$ is direct.
Also by construction
$T_r T_s = T_{r+s}$ for $r,s\in \mathbb N$.
By these comments 
the sequence $\lbrace T_n\rbrace_{n \in \mathbb N}$ is
an $\mathbb N$-grading of $T$.
\medskip

\noindent 
Fix $0 \not=q \in \mathbb F$ such that $q^2\not=1$.
Define 
\begin{eqnarray*}
\lbrack n \rbrack_q = \frac{q^n-q^{-n}}{q-q^{-1}}
\qquad \quad n \in \mathbb Z.
\end{eqnarray*}
\noindent All unadorned tensor products are meant to be
over $\mathbb F$.

\section{The algebra $U^{+}_q$}

\noindent Later in the paper we will discuss the
quantum affine algebra 
$U_q(\widehat {\mathfrak {sl}}_2)$. In the meantime
 we consider
a certain subalgebra of
$U_q(\widehat {\mathfrak {sl}}_2)$, denoted $U^+_q$ and called the positive
part of
$U_q(\widehat {\mathfrak {sl}}_2)$. 

\begin{definition} 
\label{def:Aq}
\rm
(See \cite[Corollary~3.2.6]{lusztig}.)
Let $U^+=U^+_q$ denote the $\mathbb F$-algebra with
generators $X,Y$ and relations
\begin{eqnarray}
&&
X^3 Y - \lbrack 3 \rbrack_q X^2 Y X +
\lbrack 3 \rbrack_q X Y X^2 -Y X^3 = 0,
\label{eq:XXXY}
\\
&&
Y^3 X - \lbrack 3 \rbrack_q Y^2 X Y +
\lbrack 3 \rbrack_q Y X Y^2 -X Y^3 = 0.
\label{eq:YYYX}
\end{eqnarray}
The algebra $U^+$ is called the {\it positive part of
$U_q(\widehat {\mathfrak {sl}}_2)$}. 
%%We call $X,Y$ the {\it standard generators} of $U^+$.   
The relations (\ref{eq:XXXY}),
(\ref{eq:YYYX}) are called the
{\it $q$-Serre relations}.
\end{definition}

\begin{lemma}
\label{lem:xi}
Let $a,b$ denote nonzero scalars in $\mathbb F$.
Then there exists an automorphism of $U^+$ that sends
$X \mapsto a X$ and
$Y \mapsto b Y$.
\end{lemma}
\begin{proof} Use Definition
\ref{def:Aq}.
\end{proof}

\noindent
Recall the free algebra $T$ with standard generators
$x,y$. 
There exists an $\mathbb F$-algebra homomorphism
$\mu: T \to U^+$ that sends $x\mapsto X$ and
$y\mapsto Y$. 
The homomorphism $\mu$ is surjective.
We now describe the kernel of $\mu$.
Define elements $S_x, S_y$ in $T$ by
\begin{eqnarray}
\label{eq:SgenP}
&&S_x = x^3 y - \lbrack 3 \rbrack_q x^2 y x +
\lbrack 3 \rbrack_q x y x^2 -y x^3,
\\
&&
S_y = y^3 x - \lbrack 3 \rbrack_q y^2 x y +
\lbrack 3 \rbrack_q y x y^2 -x y^3.
\label{eq:SgenM}
\end{eqnarray}
Let 
$S = T S_x T + T S_y T$
denote the 2-sided ideal of $T$ generated by
$S_x,S_y$.
Then $S$ is the kernel of $\mu$.
Recall the $\mathbb N$-grading $\lbrace T_n \rbrace_{n\in \mathbb N}$
of $T$.
For $n \in \mathbb N$ let
$U^+_n$ denote the image of 
$T_n$ under $\mu$.
The $\mathbb F$-vector space $U^+_n$ is spanned
by the products $g_1g_2\cdots g_n$ such that
$g_i$ is among $X,Y$ for $1 \leq i \leq n$.
Let $\mu_n$ denote the 
restriction of $\mu$ to $T_n$.
We view $\mu_n: T_n \to U^+_n$.
The kernel of $\mu_n $ is $S\cap T_n$. 
The subspace $S \cap T_n$ is described as follows.
Note that $S_x,S_y \in T_4$.
So for $r,s\in \mathbb N$,
$T_r S_x T_s \subseteq T_{r+s+4}$
and
$T_r S_y T_s \subseteq T_{r+s+4}$.
Consequently for $n \in \mathbb N$,
\begin{eqnarray}
S\cap T_n = 
\sum_{
\genfrac{}{}{0pt}{}{r,s \in \mathbb N}{r+s=n-4}
}
T_r S_x T_s+
\sum_{
\genfrac{}{}{0pt}{}{r,s \in \mathbb N}{r+s=n-4}
}
 T_r S_y T_s.
\label{eq:Sn}
\end{eqnarray}
Assume for the moment that $n\leq 3$.
There does not exist $r,s \in \mathbb N$ such that
$r+s=n-4$. Therefore
$S \cap T_n=0$, so
 $\mu_n: T_n \to U^+_n$ is an isomorphism.
Taking $n=0,1$ we see that
$1$ is a basis for $U^+_0$ and
$X,Y$ is a basis for $U^+_1$.
We show that the sequence
$\lbrace U^+_n \rbrace_{n \in \mathbb N}$ is an $\mathbb N$-grading
of $U^+$.
By 
(\ref{eq:Sn}) and $T=\sum_{r\in \mathbb N} T_r $
we obtain
$S=\sum_{n\in \mathbb N} (S\cap T_n)$.
Therefore the sum
$U^+ = \sum_{n \in \mathbb N} U^+_n$ is direct.
Recall that $T_r T_s = T_{r+s}$ for $r,s \in \mathbb N$.
Applying $\mu$ we find that
$U^+_r U^+_s = U^+_{r+s}$ for
$r,s\in \mathbb N$.
By these comments the sequence
$\lbrace U^+_n \rbrace_{n \in \mathbb N}$ is an $\mathbb N$-grading
of $U^+$.

\section{The $q$-Onsager algebra}

In this section we recall the $q$-Onsager algebra
and discuss its basic properties.

\begin{definition} 
\label{def:qOnsager}
\rm
(See
\cite[Section~2]{bas2},
\cite[Definition~3.9]{TwoRel}.) 
Let $\mathcal O = \mathcal O_q$ denote the $\mathbb F$-algebra with
generators $A,B$ and relations
\begin{eqnarray}
&&
A^3 B - \lbrack 3 \rbrack_q A^2 B A +
\lbrack 3 \rbrack_q A B A^2 -B A^3 = (q^2-q^{-2})^2 (BA-AB),
\label{eq:AAAB}
\\
&&
B^3 A - \lbrack 3 \rbrack_q B^2 A B +
\lbrack 3 \rbrack_q B A B^2 -A B^3 = (q^2-q^{-2})^2 (AB-BA).
\label{eq:BBBA}
\end{eqnarray}
 We call $\mathcal O$ the
 {\it $q$-Onsager algebra}.
%%We call $A,B$ the {\it standard generators} of $\mathcal O$. 
We call
(\ref{eq:AAAB}),
(\ref{eq:BBBA})
the {\it $q$-Dolan/Grady relations}.
\end{definition}

\noindent 
Consider the elements
$\lbrace A^r B^s\rbrace_{r,s\in \mathbb N}$
in the $\mathbb F$-vector space $\mathcal O$.
We show that these elements are linearly independent.
Let $A^\vee,  B^\vee$ denote commuting indeterminates.
Let $\mathbb F\lbrack  A^\vee,  B^\vee\rbrack$
denote the $\mathbb F$-algebra consisting of the polynomials
in 
$ A^\vee,  B^\vee$ that have all coefficients in 
$\mathbb F$.
The elements $\lbrace (A^\vee)^r  (B^\vee)^s\rbrace_{r,s\in \mathbb N}$
form a basis for 
the $\mathbb F$-vector space $\mathbb F\lbrack A^\vee,  B^\vee\rbrack$.
By the nature of the relations
(\ref{eq:AAAB}),
(\ref{eq:BBBA})
there exists an $\mathbb F$-algebra homomorphism
$\mathcal O \to 
\mathbb F\lbrack  A^\vee,  B^\vee\rbrack$
that sends $A \mapsto A^\vee$ and
$B \mapsto  B^\vee$.
By these comments 
the elements $\lbrace A^r B^s\rbrace_{r,s\in \mathbb N}$
in $\mathcal O$
are linearly independent.
In particular the elements $1,A,B$ in $\mathcal O$
are linearly independent.
Recall the free   algebra $T$ with standard generators
$x,y$. There exists an $\mathbb F$-algebra homomorphism
$\nu: T \to \mathcal O$ that sends $x \mapsto A$
and $y\mapsto B$. The homomorphism $\nu$ is surjective.
For $n \in \mathbb N$ let
$\mathcal O_n$ denote the image of $T_0+T_1+\cdots+T_n$ under $\nu$.
The $\mathbb F$-vector space $\mathcal O_n$ is 
spanned by the products
$g_1g_2\cdots g_r$ such that $0 \leq r \leq n$ and
$g_i$ is among $A,B$ for $1 \leq i \leq r$.
For example
$\mathcal O_0 = \mathbb F 1$ and
$\mathcal O_1 = \mathbb F 1 + \mathbb F A + \mathbb F B$.
For notational convenience define $\mathcal O_{-1} = 0$.
For $n \in \mathbb N$ we have
$\mathcal O_{n-1} + \nu(T_n)=\mathcal O_n$
and in particular $\mathcal O_{n-1} \subseteq \mathcal O_n$.
Since $\nu$ is surjective  we have 
$\mathcal O = \cup_{n \in \mathbb N} \mathcal O_n$. By
construction $\mathcal O_r \mathcal O_s = \mathcal O_{r+s}$
for $r,s \in \mathbb N$. The sequence
$\lbrace \mathcal O_n \rbrace_{n \in \mathbb N}$ 
is a filtration of $\mathcal O$ in the sense of
\cite[p.~202]{carter}.
%For $n \in \mathbb N$ let $\nu_n$ denote the
%restriction of $\nu$ to $T_n$. We view
%$\nu_n: T_n \to \mathcal O_n$.
For $n \in \mathbb N$ consider the quotient $\mathbb F$-vector
space $\overline {\mathcal O}_n
 = \mathcal O_n /\mathcal O_{n-1}$.
%%By
%%construction 
%%${\rm dim}\,\overline{\mathcal O}_n = 
%%{\rm dim} \,\mathcal O_n -
%%{\rm dim} \,\mathcal O_{n-1}$.
We view $\overline{\mathcal O}_0 =\mathcal O_0$.
The elements
$\overline{A}, \overline{B}$ form a basis for 
$\overline {\mathcal O}_1$, where
$\overline{A}=A+\mathcal O_0$ and $\overline{B}= B+\mathcal O_0$.
%%%For $u \in \mathcal O_n$ define
%%%$\overline{u} = u+\mathcal O_{n-1} \in \overline{\mathcal O}_n$.
%%The $\mathbb F$-vector space $\overline{\mathcal O}_n$ is
%%spanned by the products $\overline{g}_1 \overline{g}_2 \cdots \overline{g}_n$
%%such that $g_i$ is among $A,B$ for $1 \leq i \leq n$.
Now consider
the formal direct sum
$\overline {\mathcal O} = \sum_{n \in \mathbb N}
\overline {\mathcal O}_n$.
We emphasize that for all elements 
$u = \sum_{n \in \mathbb N} u_n$ in  
$\overline {\mathcal O}$,
the summand $u_n$ is nonzero for finitely many $n \in \mathbb N$.
By construction
$\overline{\mathcal O}$
is an $\mathbb F$-vector space.
We next define a product
$\overline {\mathcal O} \times \overline {\mathcal O} \to 
\overline {\mathcal O}$
that turns $\overline{\mathcal O}$ into an $\mathbb F$-algebra.
For $r,s\in \mathbb N$ the product sends
$\overline {\mathcal O}_r \times \overline {\mathcal O}_s
\to 
\overline {\mathcal O}_{r+s}$ as follows. 
For $u\in \mathcal O_r$ and
 $v\in \mathcal O_s$ the product of
$u+\mathcal O_{r-1}$ and
$v+\mathcal O_{s-1}$ is
$uv+\mathcal O_{r+s-1}$.
We have turned 
$\overline {\mathcal O}$ 
into an $\mathbb F$-algebra
\cite[p.~203]{carter}.
By construction
$
\overline {\mathcal O}_r
\overline {\mathcal O}_s =
\overline {\mathcal O}_{r+s}$ for $r,s \in \mathbb N$.
For $n\in \mathbb N$ the
 $\mathbb F$-vector space $\overline{\mathcal O}_n$ is
spanned by the products $\overline{g}_1 \overline{g}_2 \cdots \overline{g}_n$
such that $g_i$ is among $A,B$ for $1 \leq i \leq n$.
The $\mathbb F$-algebra 
$\overline{\mathcal O}$ is generated by
$\overline {A}, \overline {B}$.
The sequence $\lbrace 
\overline {\mathcal O}_n \rbrace_{n \in \mathbb N}$ is
an $\mathbb N$-grading of $\overline {\mathcal O}$. The
$\mathbb F$-algebra
$\overline {\mathcal O}$ is called 
the {\it graded algebra
associated with the filtration
$\lbrace \mathcal O_n \rbrace_{n \in \mathbb N}$}
\cite[p.~203]{carter}.
We now construct an $\mathbb F$-algebra homomorphism
$\overline {\nu}: T \to \overline{\mathcal O}$.
For $n \in \mathbb N$ let $\nu_n$ denote the
restriction of $\nu$ to $T_n$. We view
$\nu_n: T_n \to \mathcal O_n$. Consider the
composition
\begin{equation*}
\begin{CD} 
 \overline{\nu}_n: \quad T_n @>> \nu_n >  
\mathcal O_n 
 @>> u\mapsto u+\mathcal O_{n-1} > \overline{\mathcal O}_n.
                  \end{CD}
\end{equation*}
The map $\overline{\nu}_n:T_n\to \overline{\mathcal O}_n$
is $\mathbb F$-linear and surjective.
Define an $\mathbb F$-linear map $\overline{\nu}:T \to
\overline{\mathcal O}$ that
acts on $T_n$ as $\overline{\nu}_n$ for
$n \in \mathbb N$. 
By construction $\overline{\nu}(T_n)=\overline{\mathcal O}_n$
for $n \in \mathbb N$.
The map $\overline{\nu}$ sends
$x \mapsto \overline{A}$ and 
$y \mapsto \overline{B}$.
\begin{lemma}
The map $\overline{\nu}:T\to \overline{\mathcal O}$
is an $\mathbb F$-algebra homomorphism.
\end{lemma}
\begin{proof} Pick $r,s \in \mathbb N$. 
It suffices to show  that
$ \overline{\nu}(uv)=\overline{\nu}(u) \overline{\nu}(v)$
for $u \in T_r$ and
$v \in T_s$.
This equation is routinely verified using the definition of
the map $\overline{\nu}$ and the algebra $\overline{\mathcal O}$.
\end{proof}
\noindent The algebras $U^+$ and $\overline{\mathcal O}$ are related
as follows.
Using 
(\ref{eq:AAAB}),
(\ref{eq:BBBA}) 
we obtain
\begin{eqnarray}
&&
\overline {A}^3 \overline{B} - \lbrack 3 \rbrack_q
\overline {A}^2 \overline{B} \, \overline{A} +
\lbrack 3 \rbrack_q 
\overline{A}\,\overline{ B}\, \overline{A}^2 -
\overline{B}\, \overline{A}^3 =0,
\label{eq:AAABgraded}
\\
&&
\overline {B}^3 \overline{A} - \lbrack 3 \rbrack_q
\overline {B}^2 \overline{A}\, \overline{B} +
\lbrack 3 \rbrack_q 
\overline{B} \,\overline{ A} \,\overline{B}^2 -
\overline{A} \, \overline{B}^3 =0.
\label{eq:BBBAgraded}
\end{eqnarray}
By 
(\ref{eq:AAABgraded}),
(\ref{eq:BBBAgraded})
there exists an $\mathbb F$-algebra
homomorphism 
$\psi: U^+ \to \overline{\mathcal O}$ that sends
$X\mapsto \overline {A}$ and
$Y\mapsto \overline {B}$.
The homomorphism $\psi$ is surjective. 
By construction
$\psi(U^+_n)=
\overline{\mathcal O}_n$ for $n \in \mathbb N$.

\begin{lemma}
\label{lem:munu} 
The following diagram commutes:
\begin{equation*}
\begin{CD}
T @>I > >
                  T
           \\ 
          @V\mu VV                     @VV\overline{\nu} V \\
            U^+ @>>\psi > 
               \overline{\mathcal O} 
                   \end{CD}
\end{equation*}

\end{lemma}

\begin{proof} Each map in the diagram is an
$\mathbb F$-algebra homomorphism. To verify
that the diagram commutes, chase the standard generators
$x,y$ around the diagram.
\end{proof}

\begin{theorem}
\label{thm:main1}
The map $\psi : U^+\to \overline{\mathcal O}$ is
an isomorphism of $\mathbb F$-algebras.
\end{theorem}
\noindent The proof of Theorem
\ref{thm:main1} will be completed in
Proposition
\ref{prop:twoInj}.
\medskip

\noindent We mention one significance of
Lemma
\ref{lem:munu}  and
Theorem
\ref{thm:main1}.
Consider the free algebra $T$ with $\mathbb N$-grading
$\lbrace T_n \rbrace_{n \in \mathbb N}$. 
Pick $u \in T$. For $n \in \mathbb N$,
we say that $u$ is {\it $n$-homogeneous}
whenever $u \in T_n$. We say that $u$ is 
{\it homogeneous} whenever there exists $n \in \mathbb N$
such that $u$ is $n$-homogeneous.

\begin{definition}\rm
A subset $\Omega \subseteq T$ is said to be {\it homogeneous}
whenever each element in $\Omega$ is homogeneous.
\end{definition}

\begin{proposition}
\label{prop:Omega}
Let $\Omega$ denote a homogeneous subset of $T$ such
that the vectors
%%$\lbrace \mu(z) \rbrace_{z \in \Omega}$ 
$ \mu(z)$ $(z \in \Omega)$ 
form a basis for the $\mathbb F$-vector space $U^+$.
Then the vectors
%%$\lbrace \nu(z)\rbrace_{z \in \Omega}$ form
$\nu(z)$ $(z \in \Omega)$ form
a basis for the $\mathbb F$-vector space
$\mathcal O$.
\end{proposition}
\begin{proof}
Recall the $\mathbb F$-algebra isomorphism
$\psi: U^+ \to \overline{\mathcal O}$ from
 Theorem
\ref{thm:main1}.
For $n \in \mathbb N$ let 
 $\Omega_n$ denote the set of
 $n$-homogeneous elements in
$\Omega$.
We have $\Omega_n \subseteq T_n$ and
$\mu(T_n)=U^+_n$,
 so
$\mu(z) \in U^+_n$ for $z \in \Omega_n$.
By assumption $\Omega$ is homogeneous, so
 $\Omega = \cup_{n\in \mathbb N}\Omega_n$.
Since the vectors 
%%$\lbrace \mu(z)\rbrace_{z \in \Omega}$
$\mu(z)$ $(z \in \Omega)$
form a basis for $U^+$
and the sum
$U^+=\sum_{n \in \mathbb N} U^+_n$ is direct, we see that
for $n \in \mathbb N$ the vectors
%%$\lbrace \mu(z)\rbrace_{z \in \Omega_n}$
$\mu(z)$ $(z \in \Omega_n)$
form a basis for $U^+_n$.
We mentioned above Lemma
\ref{lem:munu} 
that 
$\psi(U^+_n) = \overline{\mathcal O}_n$,
so
the vectors
%%$\lbrace \psi(\mu(z))\rbrace_{z \in \Omega_n}$ 
$\psi(\mu(z))$ $(z \in \Omega_n)$ 
form
a basis for $\overline{\mathcal O}_n$.
By
 Lemma
\ref{lem:munu}
we have $\psi(\mu(z)) = \overline{\nu}(z)$ for $z \in \Omega_n$,
so the vectors
%%$\lbrace \overline{\nu}(z)\rbrace_{z \in \Omega_n}$
$\overline{\nu}(z)$ $(z \in \Omega_n)$
form
a basis for $\overline{\mathcal O}_n$.
Consequently the vectors 
%%$\lbrace \nu(z)\rbrace_{z \in \Omega_n}$
$\nu(z)$ $(z \in \Omega_n)$
form a basis for
a complement of
$\mathcal O_{n-1}$ in
$\mathcal O_n$.
Therefore the vectors
%%$\lbrace \nu(z)\rbrace_{z \in \Omega}$ 
$\nu(z)$ $(z \in \Omega)$ 
form a basis for $\mathcal O$.
\end{proof}

\begin{note}\rm
Assume that $\mathbb F$ is algebraically
closed with characteristic zero, and $q$ is not
a root of unity.
In 
\cite[Theorem 2.29]{shape}, T. Ito and the present
author display a homogeneous subset $\Omega$ of
$T$ such that the vectors
%%$\lbrace \mu(z)\rbrace_{z \in \Omega}$ 
$\mu(z)$ $(z \in \Omega)$ 
form a basis for
 the
$\mathbb F$-vector space $U^+$.
In \cite[Theorem~2.1]{ITaug}
the same authors show that the vectors
%%$\lbrace \nu(z)\rbrace_{z \in\Omega}$
$\nu(z)$ $(z \in\Omega)$
form a basis for the $\mathbb F$-vector space
$\mathcal O$. We point out
that
\cite[Theorem~2.1]{ITaug} follows from
\cite[Theorem 2.29]{shape} and
Proposition \ref{prop:Omega}.
\end{note}

%%%%%%%%%%%%%%%%%%%%%%%%%%%%%%%%%%%%%%%%%%%%%%%%%%%%%%%%%

\section{The algebra $\square_q$}

\noindent We have been discussing the
algebra $U^+=U^+_q$ which is the positive part of
$U_q(\widehat {\mathfrak {sl}}_2)$, and the $q$-Onsager
algebra $ \mathcal O = \mathcal O_q$. 
As we compare these algebras,
it is useful to  bring in another algebra 
$\square_q$.
In this section 
we introduce $\square_q$, and describe how it
is related to
$U^+$ and $\mathcal O$.
\medskip

\noindent Let $\mathbb Z_4 =  {\mathbb Z} /4 \mathbb Z$
denote the cyclic group of order $4$.

\begin{definition}
\rm
\label{def:boxqV1M}
Let $\square_q$ denote the $\mathbb F$-algebra with
generators $\lbrace x_i\rbrace_{i\in \mathbb Z_4}$
and relations
\begin{eqnarray}
&&
\quad \qquad \qquad 
\frac{q x_i x_{i+1}-q^{-1}x_{i+1}x_i}{q-q^{-1}} = 1,
\label{eq:centralM}
\\
&&
x^3_i x_{i+2} - \lbrack 3 \rbrack_q x^2_i x_{i+2} x_i +
\lbrack 3 \rbrack_q x_i x_{i+2} x^2_i -x_{i+2} x^3_i = 0.
\label{eq:qSerreM}
\end{eqnarray}
\end{definition}

\noindent We have some comments.

\begin{lemma}
\label{lem:aut1M}
There exists an automorphism $\rho$ of 
$\square_q$ that sends $x_i \mapsto x_{i+1}$ for
$i \in \mathbb Z_4$. Moreover $\rho^4=1$.
\end{lemma}

\begin{lemma}
\label{lem:aM}
For $0 \not=a \in \mathbb F$ there exists
an automorphism of $\square_q$ that sends
\begin{eqnarray*}
x_0 \mapsto a x_0, \qquad \quad
x_1 \mapsto a^{-1} x_1, \qquad \quad
x_2 \mapsto a x_2, \qquad \quad
x_3 \mapsto a^{-1} x_3.
\end{eqnarray*}
\end{lemma}

\noindent Our next goal is to describe how $\square_q$ is related to
$U^+$. 

\begin{definition}
\label{def:squareTeToM}
\rm Define the subalgebras 
$ \square^{\rm even}_q$,
$ \square^{\rm odd}_q$
of $ \square_q$ such that
\begin{enumerate}
\item[\rm (i)]
$ \square^{\rm even}_q$
 is 
generated by $x_0, x_2$;
\item[\rm (ii)]
$ \square^{\rm odd}_q$
  is
generated by $x_1, x_3$.
\end{enumerate}
\end{definition}

\begin{proposition}
\label{thm:tensorDecPreM}
The following {\rm (i)--(iii)} hold:
\begin{enumerate}
\item[\rm (i)] 
there exists an $\mathbb F$-algebra isomorphism
$U^+ \to  \square^{\rm even}_q$ that sends
$X\mapsto x_0$ and
$Y\mapsto x_2$;
\item[\rm (ii)] 
there exists an $\mathbb F$-algebra isomorphism
$U^+ \to  \square^{\rm odd}_q$ that sends
$X\mapsto x_1$ and
$Y\mapsto x_3$;
\item[\rm (iii)] 
the following is an isomorphism of
$\mathbb F$-vector spaces:
\begin{eqnarray*}
 \square^{\rm even}_q
\otimes
 \square^{\rm odd}_q
 & \to &  \square_q
\\
 u \otimes v  &\mapsto & uv
 \end{eqnarray*}
\end{enumerate}
\end{proposition}

\noindent The proof of Proposition
\ref{thm:tensorDecPreM}
will be completed in Section 10.
\medskip

\noindent Next we describe how
$\square_q$ is related to $\mathcal O$.

\begin{proposition}
\label{prop:ABxyiLongerM}
Pick nonzero $a,b \in \mathbb F$. 
Then there exists an $\mathbb F$-algebra 
homomorphism
$ \sharp : \mathcal O \to \square_q$ that sends
\begin{eqnarray}
\label{eq:ABDefMapM}
A \mapsto a x_0+ a^{-1} x_1, \qquad \qquad B \mapsto
b x_2+ b^{-1} x_3.
\end{eqnarray}
\end{proposition}

\noindent The proof of Proposition
\ref{prop:ABxyiLongerM}
will be completed in Section 8. 
In
Theorem \ref{thm:xiInj}
we will show that 
the map $\sharp$ from
Proposition
\ref{prop:ABxyiLongerM} is injective.

\section{The algebra $\widetilde \square_q$}

\noindent In the previous section we introduced the
algebra $\square_q$. As we investigate $\square_q$,
it is useful to consider a certain homomorphic preimage
denoted $\widetilde \square_q$. In this section
we introduce 
$\widetilde \square_q$ and describe its basic properties.

\begin{definition}
\label{def:boxqV2}
\rm
Let $\widetilde \square_q$ denote the $\mathbb F$-algebra with
generators $c^{\pm 1}_i, x_i$ $(i \in \mathbb Z_4)$
and relations
\begin{eqnarray}
&&c_i c^{-1}_i = c^{-1}_i c_i = 1,
\label{eq:cci}
\\
&&\mbox{$c^{\pm 1}_i$ is central in $\widetilde \square_q$},
\label{eq:cciCentral}
\\
&& \frac{q x_i x_{i+1}-q^{-1}x_{i+1}x_i}{q-q^{-1}} = c_i.
\label{eq:central2}
\end{eqnarray}
\end{definition}

\noindent We have some comments.

\begin{lemma}
\label{lem:aut2}
There exists an automorphism $\widetilde \rho$ of 
$\widetilde \square_q$ that sends
$c_i \mapsto c_{i+1}$ and $x_i \mapsto x_{i+1}$ for
$i \in \mathbb Z_4$. Moreover $\widetilde \rho^4=1$.
\end{lemma}

%%%%%%
%\begin{lemma}
%\label{lem:xiadjTilde}
%Let $\lbrace \alpha_i \rbrace_{i\in \mathbb Z_4}$
%denote invertible central elements in 
% $\widetilde \square_q$.
%Then there exists
%an $\mathbb F$-algebra homomorphism $\widetilde \square_q \to \widetilde 
%\square_q $ that sends
%$x_i \mapsto \alpha_i x_i $ and
%$c_i \mapsto \alpha_i \alpha_{i+1} c_i$ for $i \in \mathbb Z_4$.
%\end{lemma}
%%%%%%%%%%%%%%

\begin{lemma}
\label{lem:squareCanon}
There exists a unique $\mathbb F$-algebra homomorphism
 $ \widetilde \square_q \to \square_q$ that
 sends $c^{\pm 1}_i \mapsto 1$ and
 $x_i \mapsto x_i$ for $i \in \mathbb Z_4$.
This homomorphism is surjective.
\end{lemma}

\begin{definition}
\label{def:Canon}
\rm The homomorphism
 $ \widetilde \square_q \to \square_q$ from  
Lemma
\ref{lem:squareCanon} will be called {\it canonical}.
\end{definition}

\begin{definition}
\label{def:TeTo}
\rm Define the subalgebras 
$\widetilde \square^{\rm even}_q$,
$\widetilde \square^{\rm odd}_q$,
$\widetilde C$ of 
$\widetilde \square_q$ such that
\begin{enumerate}
\item[\rm (i)]
$\widetilde \square^{\rm even}_q$
 is 
generated by $x_0, x_2$;
\item[\rm (ii)]
$\widetilde \square^{\rm odd}_q$
  is
generated by $x_1, x_3$;
\item[\rm (iii)]
  $\widetilde C$ is
generated by $\lbrace c^{\pm 1}_i\rbrace_{i \in \mathbb Z_4}$.
\end{enumerate}
\end{definition}

%%%%%%%%%%%%%%%%%%%%%%%%%%%
%\noindent We are going to show that the following is an isomorphism
%of $\mathbb F$-vector spaces:
%\begin{eqnarray*}
%\widetilde \square^{\rm even}_q
%\otimes
%\widetilde \square^{\rm odd}_q
%\otimes
%\widetilde C
% & \to & \widetilde \square_q
%\\
% u \otimes v \otimes c  &\mapsto & uv c
% \end{eqnarray*}
%%%%%%%%%%%%%%%%%

\begin{lemma}
\label{lem:xiadjT}
Let $\lbrace \alpha_i \rbrace_{i\in \mathbb Z_4}$
denote invertible elements in 
$\widetilde C$.
Then there exists
an  $\mathbb F$-algebra homomorphism 
$\widetilde \square_q\to \widetilde \square_q $ that sends
$x_i \mapsto \alpha_i x_i $ and
$c_i \mapsto \alpha_i \alpha_{i+1} c_i$ for $i \in \mathbb Z_4$.
\end{lemma}

\begin{definition}
\label{def:alphaAutT}
\rm
The homomorphism in Lemma
\ref{lem:xiadjT} will be denoted by $\widetilde g(\alpha_0, \alpha_1, \alpha_2, \alpha_3)$.
\end{definition}

\begin{lemma} 
\label{lem:alphaAutT}
Referring to Lemma
\ref{lem:xiadjT} and Definition
\ref{def:alphaAutT}, assume that $0 \not=\alpha_i \in \mathbb F$
for $i \in \mathbb Z_4$. Then
$\widetilde g(\alpha_0, \alpha_1, \alpha_2, \alpha_3)$ is an automorphism
of $\widetilde \square_q$. Its inverse is 
$\widetilde g(\alpha^{-1}_0, \alpha^{-1}_1, \alpha^{-1}_2, \alpha^{-1}_3)$.
\end{lemma}
\begin{proof}
One checks that
$\widetilde g(\alpha_0, \alpha_1, \alpha_2, \alpha_3)$ and
$\widetilde g(\alpha^{-1}_0, \alpha^{-1}_1, \alpha^{-1}_2, \alpha^{-1}_3)$
are inverses.
Therefore
$\widetilde g(\alpha_0, \alpha_1, \alpha_2, \alpha_3)$ is invertible and
hence an automorphism of $\widetilde \square_q$.
\end{proof}

\noindent Our next goal is to obtain
an analog of
Proposition
\ref{thm:tensorDecPreM}
that applies to $\widetilde \square_q$.

\begin{lemma}
\label{lem:RR}
In $\widetilde \square_q$,
\begin{eqnarray*}
x_1 x_0 &=& q^2 x_0 x_1 + (1-q^2)c_0,
\\
x_1 x_2 &=& q^{-2} x_2 x_1 + (1-q^{-2})c_1,
\\
x_3 x_2 &=& q^2 x_2 x_3 + (1-q^2)c_2,
\\
x_3 x_0 &=& q^{-2} x_0 x_3 + (1-q^{-2})c_3.
\end{eqnarray*}
\end{lemma}
\begin{proof} These are reformulations of
(\ref{eq:central2}).
\end{proof}

\begin{definition}\rm
The four relations in Lemma
\ref{lem:RR} will be called {\it reduction rules} for
$\widetilde \square_q$.
\end{definition}

\noindent We now express the reduction rules in
a uniform way.

\begin{lemma}
\label{lem:RRuniform}
Referring to the algebra 
$\widetilde \square_q$,
pick $u \in \lbrace x_0, x_2 \rbrace$ and
 $v \in \lbrace x_1, x_3\rbrace $.
Then 
\begin{eqnarray}
\label{eq:xycom}
vu = uv q^{\langle u,v\rangle} + \gamma(u,v)(1-q^{\langle u,v\rangle})
\end{eqnarray}
where
\bigskip

\centerline{
\begin{tabular}[t]{c|cc}
$\langle\,,\,\rangle$ & $x_1$ & $x_3$
   \\  \hline
$x_0$ &
$2$ & $-2$ 
  \\ 
$x_2$ &
 $-2$ & $2$
   \\
     \end{tabular}
  \qquad \qquad
\begin{tabular}[t]{c|cc}
$\gamma(\,,\,)$ & $x_1$ & $x_3$
   \\  \hline
$x_0$ &
$c_0$ & $c_3$ 
  \\ 
$x_2$ &
 $c_1$ & $c_2$
   \\
     \end{tabular}}
\bigskip
\end{lemma}
\begin{proof} Use Lemma
\ref{lem:RR}.
\end{proof}

\begin{lemma}
\label{lem:uLong}
Fix $r \in \mathbb N$.
Referring to the algebra $\widetilde \square_q$,
pick 
$u_i \in \lbrace x_0, x_2\rbrace $ for $1 \leq i \leq r$, 
and also
$v \in \lbrace x_1, x_3 \rbrace$.
Then 
\begin{eqnarray*}
%%%%%%%%%\label{eq:longver}
&&
v u_1u_2\cdots u_r = u_1 u_2 \cdots u_r v 
q^{
\langle u_1, v\rangle
+
\cdots
+
\langle u_r, v\rangle
}
\\
&& \qquad \qquad \qquad 
+ 
\sum_{i=1}^r u_1\cdots u_{i-1} u_{i+1} \cdots u_r \gamma(u_i,v)
q^{\langle u_1,v\rangle + \cdots + \langle u_{i-1},v\rangle}
(1-q^{\langle u_i, v\rangle}).
\end{eqnarray*}
The functions $\langle \,,\,\rangle $ and
$\gamma (\,,\,)$ are from
Lemma
\ref{lem:RRuniform}.
\end{lemma}
\begin{proof} By Lemma
\ref{lem:RRuniform}
and induction on $r$.
\end{proof}

\begin{lemma} 
\label{lem:vTe}
In the algebra $\widetilde \square_q$,
\begin{eqnarray*}
&&
x_1 \widetilde \square^{\rm even}_q 
\subseteq 
\widetilde \square^{\rm even}_q  x_1 + 
\widetilde \square^{\rm even}_q  c_0 +
\widetilde \square^{\rm even}_q  c_1,
\\
&&
x_3 \widetilde \square^{\rm even}_q 
\subseteq 
\widetilde \square^{\rm even}_q  x_3 + 
\widetilde \square^{\rm even}_q  c_2 +
\widetilde \square^{\rm even}_q  c_3.
%%%v T_{\rm even} \subseteq T_{\rm even} v + T_{\rm even}C.
\end{eqnarray*}
\end{lemma}
\begin{proof}  
By Definition
\ref{def:TeTo} and
 Lemma
\ref{lem:uLong}.
\end{proof}

\begin{lemma} 
 \label{lem:factor}
 We have
 \begin{eqnarray}
 \label{eq:threeFactor}
 \widetilde \square_q = 
 \widetilde \square^{\rm even}_q 
 \widetilde \square^{\rm odd}_q 
 \widetilde C.
 \end{eqnarray}
 \end{lemma}
 \begin{proof} 
 By Definition
 \ref{def:TeTo}
 and
 Lemma
 \ref{lem:vTe} we obtain
 $
 \widetilde \square^{\rm odd}_q
 \widetilde \square^{\rm even}_q \subseteq 
 \widetilde \square^{\rm even}_q
  \widetilde \square^{\rm odd}_q
 \widetilde C$.
 By this and since $\widetilde C$ is central,
 we see that
 $
 \widetilde \square^{\rm even}_q
 \widetilde \square^{\rm odd}_q
 \widetilde C
 $
 is a subalgebra of
 $\widetilde \square_q$.
 This subalgebra contains
 $\widetilde \square^{\rm even}_q$,
 $\widetilde \square^{\rm odd}_q$,
 $\widetilde C$
  and these together generate $\widetilde \square_q$.
 The result follows.
 \end{proof}

%%%%%%%%%%
%\begin{lemma} 
% \label{lem:factor}
% We have
% \begin{eqnarray}
% \label{eq:threeFactor}
% \widetilde \square_q = 
% \widetilde \square^{\rm even}_q 
% \widetilde \square^{\rm odd}_q 
% \widetilde C.
% \end{eqnarray}
% \end{lemma}
% \begin{proof} 
% By Definition
% \ref{def:TeTo}(ii)
% and
% Lemma
% \ref{lem:vTe} we obtain
% $
% \widetilde \square^{\rm odd}_q
% \widetilde \square^{\rm even}_q \subseteq 
% \widetilde \square^{\rm even}_q
%  \widetilde \square^{\rm odd}_q
% \widetilde C$.
% By this and since $\widetilde C$ is central,
% we see that
% $
% \widetilde \square^{\rm even}_q
% \widetilde \square^{\rm odd}_q
% \widetilde C
% $
% is a subalgebra of
% $\widetilde \square_q$.
% This subalgebra contains
% $\widetilde \square^{\rm even}_q$,
% $\widetilde \square^{\rm even}_q$,
% $\widetilde C$
%  and these together generate $\widetilde \square_q$.
% The result follows.
% \end{proof}
%%%%%%%%%%%%%%%%%

\begin{definition}
\label{def:Laurent}
\rm 
Let $\lbrace \lambda_i \rbrace_{i\in \mathbb Z_4}$ denote
mutually commuting indeterminates. Let
$\mathbb F \lbrack
\lambda^{\pm 1}_0,
\lambda^{\pm 1}_1,
\lambda^{\pm 1}_2,
\lambda^{\pm 1}_3\rbrack$
denote the $\mathbb F$-algebra
consisting of the Laurent polynomials in 
$\lbrace \lambda_i \rbrace_{i\in \mathbb Z_4}$ that have all coefficients
in $\mathbb F$.
We abbreviate
\begin{eqnarray}
\label{eq:Ldef}
L= 
\mathbb F \lbrack
\lambda^{\pm 1}_0,
\lambda^{\pm 1}_1,
\lambda^{\pm 1}_2,
\lambda^{\pm 1}_3\rbrack.
\end{eqnarray}
\end{definition}

\noindent Recall the free algebra $T$ with standard generators
$x,y$.

\begin{lemma}
\label{lem:TTP}
The $\mathbb F$-vector space $T\otimes T\otimes L$
has a 
unique 
$\widetilde \square_q$-module structure 
such that for all $u,v \in T$ and $w \in L$,
\begin{enumerate}
\item[\rm (i)] 
$c^{\pm 1}_i$ sends 
$u\otimes v \otimes w \mapsto 
u\otimes v \otimes (\lambda^{\pm 1}_i w)$  for $i \in \mathbb Z_4$;
\item[\rm (ii)] 
$x_0 $ sends 
$u\otimes v \otimes w \mapsto 
(xu)\otimes v \otimes  w$;
\item[\rm (iii)] 
$x_1 $ sends 
$1\otimes v \otimes w \mapsto 
1\otimes (xv) \otimes  w$;
\item[\rm (iv)] 
$x_2 $ sends 
$u\otimes v \otimes w \mapsto 
(yu)\otimes v \otimes w$;
\item[\rm (v)] 
$x_3 $ sends 
$1\otimes v \otimes w \mapsto 
1\otimes (yv) \otimes  w$.
\end{enumerate}
\end{lemma}
%\begin{proof} Use Lemma
%\ref{lem:uLong}.
%\end{proof}
\begin{proof} We first show that
the $\widetilde \square_q$-module structure exists.
For notational convenience abbreviate $V = T \otimes T \otimes L$.
Let ${\rm End}(V)$ denote the $\mathbb F$-algebra
consisting of the $\mathbb F$-linear maps from 
$V$ to $V$. We now define some elements in
${\rm End}(V)$. Momentarily abusing notation, we call
these elements 
$\lbrace  {c}^{\pm 1}_i \rbrace_{i \in \mathbb Z_4}$
and
$\lbrace  {x}_i \rbrace_{i \in \mathbb Z_4}$.
For $i \in \mathbb Z_4$ there exist 
 ${c}^{\pm 1}_i \in
{\rm End}(V)$
that satisfy (i).
There exists 
 $x_0 \in
{\rm End}(V)$
 that satisfies (ii).
There exists 
 $x_2 \in
{\rm End}(V)$
 that satisfies (iv).
We now define 
 $x_1 \in
{\rm End}(V)$
and
$x_3 \in
{\rm End}(V)$.
To do this, we specify how
$x_1$ and  $x_3$ 
act on $u \otimes v \otimes w$.
Here we use Lemma
\ref{lem:uLong} as a guide.
Without loss of generality, we may assume that $u$ is a word in $T$.
Write $u$ as a product 
$ u_1 u_2 \cdots u_r$ of standard generators.
 The element $x_1$ sends
\begin{eqnarray*}
&&u\otimes v \otimes w \mapsto 
u \otimes (xv) \otimes  w
q^{
\langle u_1, x\rangle
+
\cdots
+
\langle u_r, x\rangle
}
\\
&& \qquad \quad
+ 
\sum_{i=1}^r u_1\cdots u_{i-1} u_{i+1} \cdots u_r 
\otimes v \otimes w
\gamma(u_i,x)
q^{\langle u_1,x\rangle + \cdots + \langle u_{i-1},x\rangle}
(1-q^{\langle u_i, x\rangle}),
\end{eqnarray*}
and $x_3 $ sends
\begin{eqnarray*}
&&u\otimes v \otimes w \mapsto 
u \otimes (yv) \otimes  w
q^{
\langle u_1, y\rangle
+
\cdots
+
\langle u_r, y\rangle
}
\\
&& \qquad \quad
+ 
\sum_{i=1}^r u_1\cdots u_{i-1} u_{i+1} \cdots u_r 
\otimes v \otimes w
\gamma(u_i,y)
q^{\langle u_1,y\rangle + \cdots + \langle u_{i-1},y\rangle}
(1-q^{\langle u_i, y\rangle}),
\end{eqnarray*}
where
\bigskip

\centerline{
\begin{tabular}[t]{c|cc}
$\langle\,,\,\rangle$ & $x$ & $y$
   \\  \hline
$x$ &
$2$ & $-2$ 
  \\ 
$y$ &
 $-2$ & $2$
   \\
     \end{tabular}
  \qquad \qquad
\begin{tabular}[t]{c|cc}
$\gamma(\,,\,)$ & $x$ & $y$
   \\  \hline
$x$ &
$\lambda_0$ & $\lambda_3$ 
  \\ 
$y$ &
 $\lambda_1$ & $\lambda_2$
   \\
     \end{tabular}}
\bigskip
\noindent We just specified how $x_1$ and $x_3$ act on
$u \otimes v \otimes w$. For $u=1$
these actions become (iii) and (v) in the 
lemma statement.
We have defined the elements
$\lbrace c^{\pm 1}_i\rbrace_{i \in \mathbb Z_4}$
and $\lbrace x_i\rbrace_{i \in \mathbb Z_4}$ in
${\rm End}(V)$.
One checks that these elements satisfy the defining relations
(\ref{eq:cci})--(\ref{eq:central2}) for
$\widetilde \square_q$.
This turns 
$V$ into
a 
$\widetilde \square_q$-module, and by construction
this 
$\widetilde \square_q$-module
 satisfies
the requirements in the lemma statement.
We have shown that
the $\widetilde \square_q$-module structure exists.
The 
 $\widetilde \square_q$-module structure is unique in
 view of Lemma
\ref{lem:uLong}.
\end{proof}

\begin{proposition}
\label{prop:tensorDec}
The following {\rm (i)--(iv)} hold:
\begin{enumerate}
\item[\rm (i)] 
there exists an $\mathbb F$-algebra isomorphism
$T\to \widetilde \square^{\rm even}_q$ that sends
$x\mapsto x_0$ and
$y\mapsto x_2$;
\item[\rm (ii)] 
there exists an $\mathbb F$-algebra isomorphism
$T\to \widetilde \square^{\rm odd}_q$ that sends
$x\mapsto x_1$ and
$y\mapsto x_3$;
\item[\rm (iii)] 
there exists an 
$\mathbb F$-algebra isomorphism
$L\to \widetilde C$ 
that sends $\lambda^{\pm 1}_i \mapsto c^{\pm 1}_i$ for
$i \in \mathbb Z_4$;
\item[\rm (iv)] 
the following is an isomorphism of
$\mathbb F$-vector spaces:
\begin{eqnarray*}
\widetilde \square^{\rm even}_q
\otimes
\widetilde \square^{\rm odd}_q
\otimes
\widetilde C
 & \to & \widetilde \square_q
\\
 u \otimes v \otimes c  &\mapsto & uv c
 \end{eqnarray*}
\end{enumerate}
\end{proposition}
\begin{proof} 
There exists a surjective $\mathbb F$-algebra
homomorphism 
$f_1:T\to \widetilde \square^{\rm even}_q$
that sends
$x \mapsto x_0$  and
$y \mapsto x_2$.
There exists a surjective $\mathbb F$-algebra
homomorphism 
$f_2:T\to \widetilde \square^{\rm odd}_q$
that sends
$x \mapsto x_1$  and
$y \mapsto x_3$.
There exists a surjective $\mathbb F$-algebra
homomorphism 
$f_3:L\to \widetilde C$
that sends
$\lambda^{\pm 1}_i \mapsto c^{\pm 1}_i$
for $i \in \mathbb Z_4$.
There exists an
$\mathbb F$-linear map $f: T\otimes T \otimes L \to 
 \widetilde \square_q$ that sends
 $u\otimes v \otimes w \mapsto f_1(u)f_2(v)f_3(w)$ for all
 $u,v \in T$ and $w \in L$. It suffices to show that $f$
 is bijective.
The map $f$ is surjective by Lemma
 \ref{lem:factor}.
 To see that
$f$ is injective, view
$T\otimes T \otimes L$ as a
$\widetilde \square_q$-module
as in Lemma
\ref{lem:TTP}.
 The map $f$ is injective because
the composition
\begin{equation*}
\begin{CD} 
T\otimes T \otimes L  @>> f >  
\widetilde \square_q  
 @>> z \mapsto z(1\otimes 1 \otimes 1) > T \otimes T \otimes L 
                  \end{CD}
\end{equation*}
is the identity map on $T\otimes T\otimes L$.
By these comments $f$ is bijective.
The result follows.
\end{proof}

%(i)--(iii) Use Lemma
%\ref{lem:TTP}.
%\\
%\noindent (iv)
%The displayed map is $\mathbb F$-linear. This
%map is surjective by
%Lemma
% \ref{lem:factor},
%and injective by
%Lemma
%\ref{lem:TTP}. It is therefore an isomorphism
%of $\mathbb F$-vector spaces.
%\end{proof}

\begin{proposition}
\label{prop:SquareBasis}
For the algebra $\widetilde \square_q$,
\begin{enumerate}
\item[\rm (i)] the following is a basis for the
$\mathbb F$-vector space $\widetilde \square^{\rm even}_q$:
\begin{eqnarray}
u_1 u_2 \cdots u_r \qquad \qquad r \in \mathbb N, \qquad
u_i \in  \lbrace x_0, x_2\rbrace, \qquad  1 \leq i \leq r;
\label{eq:NevenBasis}
\end{eqnarray}
\item[\rm (ii)] the following is a basis for the
$\mathbb F$-vector space $\widetilde \square^{\rm odd}_q$:
\begin{eqnarray}
v_1 v_2 \cdots v_s \qquad \qquad s \in \mathbb N, \qquad
v_i  \in \lbrace x_1, x_3\rbrace, \qquad 1 \leq i \leq s;
\label{eq:NoddBasis}
\end{eqnarray}
\item[\rm (iii)] the following is a basis for the
$\mathbb F$-vector space $\widetilde C$:
\begin{eqnarray}
c^{n_0}_0
c^{n_1}_1
c^{n_2}_2
c^{n_3}_3,
\qquad \qquad n_i \in \mathbb Z, \qquad i \in \mathbb Z_4;
\label{eq:Cbasis}
\end{eqnarray}
\item[\rm (iv)] the
$\mathbb F$-vector space $\widetilde \square_q$ has a basis consisting
of the elements
\begin{eqnarray}
u 
v
c
\label{eq:basisWWC}
\end{eqnarray}
such that 
$u$, 
$v$, 
$c$ are contained in the bases
$(\ref{eq:NevenBasis})$,
$(\ref{eq:NoddBasis})$,
$(\ref{eq:Cbasis})$, respectively.
\end{enumerate}
\end{proposition}
\begin{proof}
By Proposition
\ref{prop:tensorDec}.
\end{proof}

\section{Some calculations in $\widetilde \square_q$}

\noindent We continue to investigate the algebra $\widetilde \square_q$
from Definition
\ref{def:boxqV2}.
In this section, we obtain some results about
$\widetilde \square_q$
that will be used
in the proof of
Proposition
\ref{thm:tensorDecPreM}. 

\begin{definition}
\label{def:Deltai}
\rm For the algebra $\widetilde \square_q$ define
\begin{eqnarray}
S_i = x^3_ix_{i+2} -
\lbrack 3 \rbrack_q  x^2_i x_{i+2}x_i
+
\lbrack 3 \rbrack_q  x_i x_{i+2}x^2_i
-
x_{i+2}x^3_i \qquad \qquad i \in \mathbb Z_4.
\label{eq:Delta}
\end{eqnarray}
\end{definition}

\noindent We note that the elements $S_i$ from 
Definition
\ref{def:Deltai} are in the kernel of the canonical 
homomorphism 
$\widetilde \square_q \to \square_q$.

%%%%%%%%%move below to next section
%\noindent In Proposition
%\ref{prop:tensorDec}
%we obtained a tensor product decomposition
%of the algebra $\widetilde \square_q$.
%In this section we obtain  an analogous decomposition
%for $\square_q$.
%%%%%%%move to next section

\begin{lemma}
We have
$S_0, S_2 \in \widetilde \square^{\rm even}_q$
and
$S_1, S_3 \in \widetilde \square^{\rm odd}_q$.
\end{lemma}
\begin{proof} By Definition
\ref{def:TeTo}(i),(ii) and
Definition
\ref{def:Deltai}.
\end{proof}

\begin{proposition}
\label{lem:deltaCom}
In the algebra $\widetilde \square_q$ the following
{\rm (i), (ii)} hold for
$i \in \mathbb Z_4$:
\begin{enumerate}
\item[\rm (i)]
$ x_{i+1} S_i =
q^4 S_i x_{i+1}$, 
\item[\rm (ii)]
$x_{i-1} S_i =
q^{-4} S_i x_{i-1}$.
\end{enumerate}
\end{proposition}
\begin{proof}(i) Without loss of generality we may assume
that $i=0$. Our strategy is to
express 
$
 x_1 S_0-
q^4 S_0 x_1$
in the basis
for $\widetilde \square_q$ from Proposition
\ref{prop:SquareBasis}(iv).
By Definition
\ref{def:Deltai},
\begin{eqnarray}
x_1 S_0 = x_1 x^3_0 x_2 - \lbrack 3 \rbrack_q x_1x^2_0 x_2 x_0 + 
\lbrack 3 \rbrack_q x_1x_0 x_2 x^2_0 - x_1 x_2 x^3_0,
\label{ex:x1Delta0}
\\
S_0 x_1 =  x^3_0 x_2 x_1- \lbrack 3 \rbrack_q x^2_0 x_2 x_0 x_1+ 
\lbrack 3 \rbrack_q x_0 x_2 x^2_0 x_1 -  x_2 x^3_0 x_1.
\label{ex:Delta0x1}
\end{eqnarray}
Consider the elements
\begin{eqnarray}
\label{ex:DeltaTerm}
x_1 x^3_0 x_2,
\qquad
x_1 x^2_0 x_2 x_0,
\qquad
x_1 x_0 x_2 x^2_0,
\qquad
x_1 x_2 x^3_0,
\qquad
S_0 x_1.
\end{eqnarray}
By 
Lemma \ref{lem:uLong} and
(\ref{ex:Delta0x1}),
the elements
(\ref{ex:DeltaTerm}) are weighted sums involving 
the following terms and coefficients:
\medskip

\centerline{
\begin{tabular}[t]{c|ccccc}
{\rm term} & {\rm $x_1x^3_0x_2$ coef.} & {\rm $x_1x^2_0x_2x_0$ coef.} & 
{\rm $x_1x_0x_2x^2_0$ coef.}  & {\rm $x_1x_2x^3_0$ coef.} & 
{\rm $S_0 x_1$ coef.} 
   \\ \hline \hline 
$x^3_0 x_2 x_1 $   & $q^4$ & $0$ & $0$ & $0$ &  $1$
\\
$x^2_0 x_2 x_0x_1$     & $0$   & $q^4$ & $0$ & $0$ & $-\lbrack 3 \rbrack_q$
\\
$x_0   x_2 x^2_0x_1 $  & $0$ & $0$ & $q^4$ & $0$ & $\lbrack 3 \rbrack_q$ 
\\
$      x_2 x^3_0 x_1 $  &$0$ & $0$ &$0$ & $q^4$ & $-1$
\\
\hline
$x^2_0 x_2 c_0$     & $1-q^6$   & $q^2-q^4$ & $0$ & $0$ & $0$
\\
$x_0 x_2 x_0 c_0$ & $0$    & $1-q^4$   & $1-q^4$ & $0$ & $0$
\\
$ x_2 x^2_0 c_0$ & $0$ &  $0$    & $1-q^2$   & $q^{-2}-q^4$ & $0$ 
\\
\hline
$ x^3_0 c_1$ & $q^6-q^4$ &  $q^4-q^2$    & $q^2-1$   & $1-q^{-2}$ & $0$ 
\\
\end{tabular}}
\bigskip

\noindent 
We can now easily write
$x_1S_0 - q^4 S_0 x_1$  in the basis for
$\widetilde \square_q$ from Proposition
\ref{prop:SquareBasis}(iv).
To do this,
expand 
$x_1 S_0 - q^4 S_0 x_1$ 
using
(\ref{ex:x1Delta0})
and evaluate
the result using the 
 data in the above table.
After a routine cancellation we obtain
$x_1 S_0 - q^4 S_0 x_1=0$.
Therefore 
$x_1 S_0 =q^4 S_0 x_1$.
\\
\noindent (ii) We proceed as in part (i) above.
Without loss of generality we may assume
that $i=0$. We express
$
 x_3 S_0-
q^{-4} S_0 x_3$
in the basis
for $\widetilde \square_q$ from Proposition
\ref{prop:SquareBasis}(iv).
By Definition
\ref{def:Deltai},
\begin{eqnarray}
x_3 S_0 = x_3 x^3_0 x_2 - \lbrack 3 \rbrack_q x_3x^2_0 x_2 x_0 + 
\lbrack 3 \rbrack_q x_3x_0 x_2 x^2_0 - x_3 x_2 x^3_0,
\label{ex:x3Delta0}
\\
S_0 x_3 =  x^3_0 x_2 x_3- \lbrack 3 \rbrack_q x^2_0 x_2 x_0 x_3+ 
\lbrack 3 \rbrack_q x_0 x_2 x^2_0 x_3 -  x_2 x^3_0 x_3.
\label{ex:Delta0x3}
\end{eqnarray}
Consider the elements
\begin{eqnarray}
\label{ex:DeltaTermx3}
x_3 x^3_0 x_2,
\qquad
x_3 x^2_0 x_2 x_0,
\qquad
x_3 x_0 x_2 x^2_0,
\qquad
x_3 x_2 x^3_0,
\qquad
S_0 x_3.
\end{eqnarray}
By 
Lemma \ref{lem:uLong} and
(\ref{ex:Delta0x3}),
the elements
(\ref{ex:DeltaTermx3}) are weighted sums involving
the following terms and coefficients:
\medskip

\centerline{
\begin{tabular}[t]{c|ccccc}
{\rm term} & {\rm $x_3x^3_0x_2$ coef.} & {\rm $x_3x^2_0x_2x_0$ coef.} & 
{\rm $x_3x_0x_2x^2_0$ coef.}  & {\rm $x_3x_2x^3_0$ coef.} & 
{\rm $S_0 x_3$ coef.} 
   \\ \hline \hline 
$x^3_0 x_2 x_3 $   & $q^{-4}$ & $0$ & $0$ & $0$ &  $1$
\\
$x^2_0 x_2 x_0x_3$     & $0$   & $q^{-4}$ & $0$ & $0$ & $-\lbrack 3 \rbrack_q$
\\
$x_0   x_2 x^2_0x_3 $  & $0$ & $0$ & $q^{-4}$ & $0$ & $\lbrack 3 \rbrack_q$ 
\\
$      x_2 x^3_0 x_3 $  &$0$ & $0$ &$0$ & $q^{-4}$ & $-1$
\\
\hline
$x^2_0 x_2 c_3$     & $1-q^{-6}$   & $q^{-2}-q^{-4}$ & $0$ & $0$ & $0$
\\
$x_0 x_2 x_0 c_3$ & $0$    & $1-q^{-4}$   & $1-q^{-4}$ & $0$ & $0$
\\
$ x_2 x^2_0 c_3$ & $0$ &  $0$    & $1-q^{-2}$   & $q^{2}-q^{-4}$ & $0$ 
\\
\hline
$ x^3_0 c_2$ & $q^{-6}-q^{-4}$ &  $q^{-4}-q^{-2}$    & $q^{-2}-1$   & $1-q^{2}$ & $0$ 
\\
\end{tabular}}
\bigskip

\noindent 
Now to write
$x_3 S_0 - q^{-4} S_0 x_3$  in the basis for
$\widetilde \square_q$ from Proposition
\ref{prop:SquareBasis}(iv),
expand 
$x_3 S_0 - q^{-4} S_0 x_3$ 
using
(\ref{ex:x3Delta0})
and evaluate
the result using the 
 data in the above table.
After a routine cancellation we obtain
$x_3 S_0 - q^{-4} S_0 x_3=0$.
Therefore 
$x_3 S_0 =q^{-4} S_0 x_3$.
\end{proof}

\begin{corollary}
\label{cor:DeltaCom}
In the algebra $\widetilde \square_q$,
\begin{eqnarray*}
&&
S_0 \widetilde \square^{\rm odd}_q 
=
\widetilde \square^{\rm odd}_q S_0 ,
\qquad \qquad 
S_1 \widetilde \square^{\rm even}_q 
=
\widetilde \square^{\rm even}_q S_1,
\\
&&
S_2 \widetilde \square^{\rm odd}_q 
=
\widetilde \square^{\rm odd}_q S_2,
\qquad \qquad 
S_3 \widetilde \square^{\rm even}_q 
=
\widetilde \square^{\rm even}_q S_3.
\end{eqnarray*}
\end{corollary}
\begin{proof} By Definition
\ref{def:TeTo}(i),(ii)
and
Proposition
\ref{lem:deltaCom}.
\end{proof}

\section{More calculations in $\widetilde \square_q$
}

\noindent We continue to
investigate the algebra 
$\widetilde \square_q$ from
Definition
\ref{def:boxqV2}.
 In this section, we first obtain some results
about $\widetilde \square_q$. We then use these results
to prove
Proposition
\ref{prop:ABxyiLongerM}.
\medskip

\noindent 
For the algebra $\widetilde \square_q$ define
\begin{eqnarray}
A = x_0+x_1, \qquad \qquad B = x_2+x_3.
\label{eq:AB}
\end{eqnarray}
Our next goal is to express 
\begin{eqnarray}
&&
A^3 B - \lbrack 3 \rbrack_q A^2 B A +
\lbrack 3 \rbrack_q A B A^2 -B A^3 + 
(q^2-q^{-2})^2 c_0(AB-BA)
\label{eq:terms}
\end{eqnarray}
in the basis 
for $\widetilde \square_q$ from
Proposition
\ref{prop:SquareBasis}(iv).
%%%%%%%%%(\ref{eq:basisWWC}).
\medskip

\noindent 
For the rest of this section, the notation 
(\ref{eq:AB}) 
is in effect. 

\begin{lemma}
\label{lem:AAAB}
In the algebra $\widetilde \square_q$,
\begin{enumerate}
\item[\rm (i)] 
the element $A^2$ is a weighted sum involving the following
terms and coefficients:
\medskip

\centerline{
\begin{tabular}[t]{c|cccc}
{\rm term} & $x^2_0$ & $x_0x_1$ & $x^2_1$ & $c_0$
   \\ \hline 
{\rm coefficient} & $1$ & $1+q^2$ & $1$ & $1-q^2$
\\
     \end{tabular}}
\item[\rm (ii)] 
the element $AB$ is a weighted sum involving the following
terms and coefficients:
\medskip

\centerline{
\begin{tabular}[t]{c|ccccc}
{\rm term} & $x_0x_2$ & $x_0x_3$ & $x_2x_1$ & $x_1x_3$ & $c_1$
   \\ \hline 
{\rm coefficient} & $1$ & $1$ & $q^{-2}$ & $1$ & $1-q^{-2}$
\\
     \end{tabular}}
\item[\rm (iii)] 
the element $BA$ is a weighted sum involving the following
terms and coefficients:
\medskip

\centerline{
\begin{tabular}[t]{c|ccccc}
{\rm term} & $x_2x_0$ & $x_0x_3$ & $x_2x_1$ & $x_3x_1$ & $c_3$
   \\ \hline 
{\rm coefficient} & $1$ & $q^{-2}$ & $1$ & $1$ & $1-q^{-2}$
\\
     \end{tabular}}

\medskip
\end{enumerate}
\end{lemma}
\begin{proof} Use Lemma
\ref{lem:RR} and
(\ref{eq:AB}).
\end{proof}

\noindent Next we use Lemma
\ref{lem:AAAB} to 
evaluate some terms in (\ref{eq:terms}).
Consider $A^3B=(A^2)(AB)$.
In this equation, evaluate the right-hand side
using
Lemma \ref{lem:AAAB}(i),(ii) and expand the result;
the details are 
 in the table below.
The table has two header columns that describe $A^2$,
and two header rows that describe $AB$.
The expressions inside  parentheses
are ``out of order''
and will be subject to further reduction shortly.
\begin{eqnarray*} \qquad \qquad \qquad  \qquad \qquad AB
\end{eqnarray*}
\centerline{ 
\begin{tabular}[t]{c}
\\
\\
\\
$A^2$
\\
\\
     \end{tabular}
\qquad
\begin{tabular}[t]{cc|ccccc}
& & $1$ & $1$ & $q^{-2}$ & $1$ & $1-q^{-2}$ 
\\
& & $x_0x_2$ & $x_0x_3$ & $x_2x_1$ & $x_1x_3$ & $c_1$ 
   \\ \hline 
$1$ & $x^2_0$
 & $x^3_0x_2$ & $x^3_0x_3$ & $x^2_0x_2x_1$ & $x^2_0x_1x_3$ & $x^2_0c_1$ 
\\
$1+q^2$ & $x_0x_1$ &
 $x_0(x_1 x_0x_2)$ & $x_0(x_1 x_0)x_3$ & $x_0(x_1x_2)x_1$ & $x_0x^2_1x_3$ &
 $x_0x_1c_1$ 
\\
$1$ & $x^2_1$ &
 $(x^2_1x_0x_2)$ & $(x^2_1x_0)x_3$ & $(x^2_1x_2)x_1$ & $x^3_1x_3$ & $x^2_1c_1$ 
\\
$1-q^2$ & $c_0$ &
$x_0x_2c_0$ & $x_0x_3c_0$ & $x_2x_1c_0$ & $x_1x_3c_0$ & $c_0c_1$ 
\\
     \end{tabular}}
\bigskip

\noindent Similarly for
$A^2BA=(A^2)(BA)$,
\begin{eqnarray*} \qquad \qquad \qquad  \qquad \qquad BA
\end{eqnarray*}
\centerline{ 
\begin{tabular}[t]{c}
\\
\\
\\
$A^2$
\\
\\
     \end{tabular}
\qquad
\begin{tabular}[t]{cc|ccccc}
& & $1$ & $q^{-2}$ & $1$ & $1$ & $1-q^{-2}$ 
\\
& & $x_2x_0$ & $x_0x_3$ & $x_2x_1$ & $x_3x_1$ & $c_3$ 
   \\ \hline 
$1$ & $x^2_0$
 & $x^2_0x_2x_0$ & $x^3_0x_3$ & $x^2_0x_2x_1$ & $x^2_0 x_3x_1$ & $x^2_0c_3$ 
\\
$1+q^2$ & $x_0x_1$ &
 $x_0(x_1 x_2x_0)$ & $x_0(x_1 x_0)x_3$ & $x_0(x_1x_2)x_1$ & $x_0x_1x_3x_1$ &
 $x_0x_1c_3$ 
\\
$1$ & $x^2_1$ &
 $(x^2_1x_2x_0)$ & $(x^2_1x_0)x_3$ & $(x^2_1x_2)x_1$ & $x^2_1x_3x_1$ & $x^2_1c_3$ 
\\
$1-q^2$ & $c_0$ &
$x_2x_0c_0$ & $x_0x_3c_0$ & $x_2x_1c_0$ & $x_3x_1c_0$ & $c_0c_3$ 
\\
     \end{tabular}}
\bigskip

\noindent Similarly for
$ABA^2=(AB)(A^2)$,
\begin{eqnarray*} \qquad  \qquad \qquad \qquad  \qquad \qquad A^2
\end{eqnarray*}
\centerline{ 
\begin{tabular}[t]{c}
\\
\\
\\
\\
$AB$
\\
\\
     \end{tabular}
\qquad
\begin{tabular}[t]{cc|cccc}
& & $1$ & $1+q^2$ & $1$ &  $1-q^2$ 
\\
& & $x^2_0$ & $x_0x_1$ & $x^2_1$ &  $c_0$ 
   \\ \hline 
$1$ & $x_0x_2$
 & $x_0x_2x^2_0$ & $x_0x_2x_0x_1$ & $x_0x_2x^2_1$ & $x_0x_2c_0$ 
\\
$1$ & $x_0x_3$ &
 $x_0(x_3 x^2_0)$ & $x_0(x_3 x_0)x_1$ & $x_0x_3x^2_1$ &
 $x_0x_3c_0$ 
\\
$q^{-2}$ & $x_2x_1$ &
 $x_2(x_1x^2_0)$ & $x_2(x_1x_0)x_1$ & $x_2x^3_1$ & $x_2x_1c_0$ 
\\
$1$ & $x_1x_3$ &
$(x_1x_3x^2_0)$ & $(x_1x_3x_0)x_1$ & $x_1x_3x^2_1$ & $x_1x_3c_0$ 
\\
$1-q^{-2}$ & $c_1$  &
$x^2_0c_1$ & $x_0x_1c_1$ & $x^2_1c_1$ & $c_0c_1$ 
\\
\end{tabular}}
\bigskip

\noindent Similarly for
$BA^3=(BA)(A^2)$,
\begin{eqnarray*} \qquad  \qquad \qquad \qquad  \qquad \qquad A^2
\end{eqnarray*}
\centerline{ 
\begin{tabular}[t]{c}
\\
\\
\\
\\
$BA$
\\
\\
     \end{tabular}
\qquad
\begin{tabular}[t]{cc|cccc}
& & $1$ & $1+q^2$ & $1$ &  $1-q^2$ 
\\
& & $x^2_0$ & $x_0x_1$ & $x^2_1$ &  $c_0$ 
   \\ \hline 
$1$ & $x_2x_0$
 & $x_2x^3_0$ & $x_2x^2_0x_1$ & $x_2x_0x^2_1$ & $x_2x_0c_0$ 
\\
$q^{-2}$ & $x_0x_3$ &
 $x_0(x_3 x^2_0)$ & $x_0(x_3 x_0)x_1$ & $x_0x_3x^2_1$ &
 $x_0x_3c_0$ 
\\
$1$ & $x_2x_1$ &
 $x_2(x_1x^2_0)$ & $x_2(x_1x_0)x_1$ & $x_2x^3_1$ & $x_2x_1c_0$ 
\\
$1$ & $x_3x_1$ &
$(x_3x_1x^2_0)$ & $(x_3x_1x_0)x_1$ & $x_3x^3_1$ & $x_3x_1c_0$ 
\\
$1-q^{-2}$ & $c_3$  &
$x^2_0c_3$ & $x_0x_1c_3$ & $x^2_1c_3$ & $c_0c_3$ 
\\
\end{tabular}}
\bigskip

\noindent The above four tables 
contain some 
parenthetical expressions.
We will write these parenthetical expressions in the
 basis for
$\widetilde \square_q$ from
Proposition
\ref{prop:SquareBasis}(iv).
For the parenthetical expressions of length two,
this is done using Lemma
\ref{lem:RR}.
For the remaining parenthetical expressions,
this will be done over the next three lemmas.

\begin{lemma}
\label{lem:NTL1}
In the algebra $\widetilde \square_q$,
\begin{enumerate}
\item[\rm (i)] 
the element $x_1x_0x_2$ is a weighted sum involving the following
terms and coefficients:
\medskip

\centerline{
\begin{tabular}[t]{c|ccc}
{\rm term} & $x_0x_2x_1$ & $x_0c_1$ & $x_2c_0$ 
   \\ \hline 
{\rm coefficient} & $1$ & $q^2-1$ & $1-q^2$ 
\\
     \end{tabular}}
\item[\rm (ii)] 
the element $x^2_1x_0$ is a weighted sum involving the following
terms and coefficients:
\medskip

\centerline{
\begin{tabular}[t]{c|cc}
{\rm term} & $x_0x^2_1$ & $x_1c_0$ 
   \\ \hline 
{\rm coefficient} & $q^4$ & $1-q^4$ 
\\
     \end{tabular}}
\item[\rm (iii)] 
the element $x^2_1x_2$ is a weighted sum involving the following
terms and coefficients:
\medskip

\centerline{
\begin{tabular}[t]{c|cc}
{\rm term} & $x_2x^2_1$ & $x_1c_1$ 
   \\ \hline 
{\rm coefficient} & $q^{-4}$ & $1-q^{-4}$ 
\\
     \end{tabular}}

\item[\rm (iv)] 
the element $x_1x_2x_0$ is a weighted sum involving the following
terms and coefficients:
\medskip

\centerline{
\begin{tabular}[t]{c|ccc}
{\rm term} & $x_2x_0x_1$ & $x_2c_0$ & $x_0c_1$ 
   \\ \hline 
{\rm coefficient} & $1$ & $q^{-2}-1$ & $1-q^{-2}$ 
\\
     \end{tabular}}

\medskip
\end{enumerate}
\end{lemma}
\begin{proof} Apply Lemma
\ref{lem:RR} repeatedly.
\end{proof}

\begin{lemma}
\label{lem:NTL2}
In the algebra $\widetilde \square_q$,
\begin{enumerate}
\item[\rm (i)] 
the element $x_1x_3x_0$ is a weighted sum involving the following
terms and coefficients:
\medskip

\centerline{
\begin{tabular}[t]{c|ccc}
{\rm term} & $x_0x_1x_3$ & $x_1c_3$ & $x_3c_0$ 
   \\ \hline 
{\rm coefficient} & $1$ &  $1-q^{-2}$ & 
$q^{-2}-1$ 
\\
     \end{tabular}}
\item[\rm (ii)] 
the element $x_1x^2_0$ is a weighted sum involving the following
terms and coefficients:
\medskip

\centerline{
\begin{tabular}[t]{c|cc}
{\rm term} & $x^2_0x_1$ & $x_0c_0$ 
   \\ \hline 
{\rm coefficient} & $q^4$ & $1-q^4$ 
\\
     \end{tabular}}
\item[\rm (iii)] 
the element $x_3x^2_0$ is a weighted sum involving the following
terms and coefficients:
\medskip

\centerline{
\begin{tabular}[t]{c|cc}
{\rm term} & $x^2_0x_3$ & $x_0c_3$ 
   \\ \hline 
{\rm coefficient} & $q^{-4}$ & $ 1-q^{-4}$ 
\\
     \end{tabular}}

\item[\rm (iv)] 
the element $x_3x_1x_0$ is a weighted sum involving the following
terms and coefficients:
\medskip

\centerline{
\begin{tabular}[t]{c|ccc}
{\rm term} & $x_0x_3x_1$ & $x_1c_3$ & $x_3c_0$ 
   \\ \hline 
{\rm coefficient} & $1$ & $q^2-1$ & $1-q^2$ 
\\
     \end{tabular}}

\medskip
\end{enumerate}
\end{lemma}
\begin{proof} Similar to the proof of
Lemma \ref{lem:NTL1}.
\end{proof}

\begin{lemma}
\label{lem:NTL3}
In the algebra $\widetilde \square_q$,
\begin{enumerate}
\item[\rm (i)] 
the element $x^2_1x_0x_2$ is a weighted sum involving the following
terms and coefficients:
\medskip

\centerline{
\begin{tabular}[t]{c|cccc}
{\rm term} & $x_0x_2x^2_1$ & $x_0x_1c_1$ & $x_2x_1c_0$  & $c_0c_1$
   \\ \hline 
{\rm coefficient} &
$1$ &  $q^4-1$ & 
$q^{-2}-q^{2}$ &
$(1-q^2)(q^2-q^{-2})$ 
\\
     \end{tabular}}
\item[\rm (ii)] 
the element $x^2_1x_2x_0$ is a weighted sum involving the following
terms and coefficients:
\medskip

\centerline{
\begin{tabular}[t]{c|cccc}
{\rm term} & $x_2x_0 x^2_1$ & $x_0x_1c_1$  & $x_2x_1c_0$ & $c_0c_1$
   \\ \hline 
{\rm coefficient} & $1$ & $q^2-q^{-2}$ & 
$q^{-4}-1$ & $(q^{-2}-1)(q^2-q^{-2})$ 
\\
     \end{tabular}}
\item[\rm (iii)] 
the element $x_1x_3x^2_0$ is a weighted sum involving the following
terms and coefficients:
\medskip

\centerline{
\begin{tabular}[t]{c|cccc}
{\rm term} & $x^2_0x_1x_3$ & $x_0x_1c_3$ & $x_0x_3c_0$ & $c_0c_3$ 
   \\ \hline 
{\rm coefficient} & $1$ & $q^2-q^{-2}$ &
$q^{-4}-1$ &
$(q^{-2}-1)(q^2-q^{-2})$
\\
     \end{tabular}}

\item[\rm (iv)] 
the element $x_3x_1x^2_0$ is a weighted sum involving the following
terms and coefficients:
\medskip

\centerline{
\begin{tabular}[t]{c|cccc}
{\rm term} & $x^2_0x_3x_1$ & $x_0x_1c_3$ & $x_0x_3c_0$  & $c_0c_3$
   \\ \hline 
{\rm coefficient} & $1$ & $q^4-1$ & $q^{-2}-q^{2}$  
 & $ (1-q^2)(q^2-q^{-2})$ 
\\
     \end{tabular}}

\medskip
\end{enumerate}
\end{lemma}
\begin{proof} Similar to the proof of
Lemma \ref{lem:NTL1}.
\end{proof}

\noindent Referring to the algebra
$\widetilde \square_q$, we now write 
the elements
\begin{eqnarray}
A^3B, \qquad A^2BA, \qquad ABA^2, \qquad BA^3
\label{eq:fiveTerms}
\end{eqnarray}
in the basis for $\widetilde \square_q$ from Proposition
\ref{prop:SquareBasis}(iv).

\begin{lemma}
\label{lem:expand}
In the algebra $\widetilde \square_q$, the elements
{\rm (\ref{eq:fiveTerms})} are weighted sums involving
the following terms and coefficients.

\medskip

\centerline{
\begin{tabular}[t]{c|ccccc}
{\rm term} & {\rm $A^3B$ coef.} & {\rm $A^2BA$ coef.} & 
{\rm $ABA^2$ coef.}  & {\rm $BA^3$ coef.} 
   \\ \hline \hline 
$x^3_0x_2$ &      $1$ & $0$ & $0$ & $0$ 
\\
$x^2_0x_2x_0$  &      $0$ & $1$ & $0$ &$0$
\\
$x_0x_2x^2_0$  &      $0$ & $0$ & $1$ & $0$ 
\\
$x_2x^3_0$  &      $0$ & $0$ & $0$ & $1$
\\
\hline
$x^3_1x_3$ &      $1$ & $0$ & $0$ & $0$
\\
$x^2_1x_3x_1$ &      $0$ & $1$ & $0$ & $0$ 
\\
$x_1x_3x^2_1$ &      $0$ & $0$ & $1$ & $0$
\\
$x_3x^3_1$ &      $0$ & $0$ & $0$ & $1$ 
\\
\hline
$x^3_0x_3$ &      $1$ & $q^{-2}$ & $q^{-4}$ & $q^{-6}$ 
\\
$x_2x^3_1$ &      $q^{-6}$ & $q^{-4}$ & $q^{-2}$ & $1$ 
\\
\hline
$x_0x^2_1x_3$ &      $q^{2}\lbrack 3\rbrack_q$ & $q^{2}$ & $0$ & $0$ 
\\
$x_0x_1x_3x_1$ &    $0$ & $q^{2}+1$ & $q^{2}+1$ &  $0$ 
\\
$x_0x_3x^2_1$ &   $0$ &  $0$ & $1$ & $\lbrack 3 \rbrack_q$
\\
\hline
$x^2_0x_2x_1$ &   $\lbrack 3 \rbrack_q$ &  $1$ & $0$ & $0$ 
\\
$x_0x_2x_0x_1$ &    $0$ & $q^{2}+1$ & $q^{2}+1$ &  $0$ 
\\
$x_2x^2_0x_1$ &   $0$ &  $0$ & $q^2$ & $q^2 \lbrack 3 \rbrack_q$  
\\
\hline
$x^2_0x_1x_3$ &      $q^{2}\lbrack 3\rbrack_q$ & $q^{2}+1$ & $1$ & $0$ 
\\
$x^2_0x_3x_1$ &   $0$ &  $1$ & $q^{-2}+1$ & $q^{-2}\lbrack 3 \rbrack_q$ 
\\
\hline
$x_0x_2x^2_1$ &  $q^{-2}\lbrack 3\rbrack_q$ & $q^{-2}+1$ & $1$ & $0$ 
\\
$x_2x_0x^2_1$ &   $0$ &  $1$ & $q^{2}+1$ & $q^{2}\lbrack 3 \rbrack_q$ 
\\
\end{tabular}}
\bigskip

\centerline{
\begin{tabular}[t]{c|cccc}
{\rm term} & {\rm $A^3B$ coef.} & {\rm $A^2BA$ coef.} & 
{\rm $ABA^2$ coef.}  & {\rm $BA^3$ coef.} 
   \\ \hline \hline 
$x^2_0 c_1$ &         $(q^2-1)\lbrack 3 \rbrack_q$ & $q^2-q^{-2}$ & $1-q^{-2}$ & $0$
\\
$x^2_0 c_3$ &         $0$ & $1-q^{-2}$ & $1-q^{-4}$ & 
$q^{-2}(1-q^{-2})\lbrack 3 \rbrack_q $ 
\\
$x_0x_2 c_0$ &         $(1-q^2)(2+q^2)$ & $q^{-2}-q^2$ & $1-q^2$ & $0$ 
\\
$x_2x_0 c_0$ &         $0$ & $1-q^2$ & $q^{-2}-q^2$ & $(1-q^2)(2+q^2)$ 
\\
\hline
$x_0x_1 c_1$ &         $(q^2-q^{-2})\lbrack 3 \rbrack_q$ & $2(q^2-q^{-2})$ & 
$q^2-q^{-2}$ & $0$ 
\\
$x_0x_1 c_3$ &         $0$ & $q^2-q^{-2}$ & $2(q^2-q^{-2})$
& $(q^2-q^{-2})\lbrack 3 \rbrack_q$ 
\\
$x_0x_3 c_0$ &         $(1-q^2)(2+q^2)$ & $(q^{-2}-1)(q^2+2)$ &
$(q^{-2}-1)\lbrack 3 \rbrack_q$ & $(q^{-2}-1)(q^2+2)$ 
\\
$x_2x_1 c_0$ &         $(q^{-2}-1)(2+q^2)$ & $(q^{-2}-1)\lbrack 3\rbrack_q$ 
& $(q^{-2}-1)(q^2+2)$ & $(1-q^2)(2+q^2)$ 
\\
\hline
$x^2_1 c_1$ &         $q^{-2}(1-q^{-2})\lbrack 3 \rbrack_q$ & $1-q^{-4}$ &
$1-q^{-2}$& $0$ 
\\
$x^2_1 c_3$ &         $0$ & $1-q^{-2}$ & $q^2-q^{-2}$ & $(q^2-1)\lbrack 3 \rbrack_q$ 
\\
$x_1x_3 c_0$ &         $(1-q^2)(2+q^2)$ & $q^{-2}-q^2$ & $1-q^2$ & $0$ 
\\
$x_3x_1 c_0$ &         $0$ & $1-q^{2}$ & $q^{-2}-q^2$ & $(1-q^2)(2+q^2)$ 
\\
\hline
$ c_0c_1$ &         $-(q-q^{-1})^2(2+q^2)$ & $(q^{-2}-1)(q^2-q^{-2})$ 
& $-(q-q^{-1})^2$ & $0$ 
\\
$ c_0c_3$ &         $0$ & $-(q-q^{-1})^2$ & $(q^{-2}-1)(q^2-q^{-2})$ & $-(q-q^{-1})^2(2+q^2)$ 
\\
\end{tabular}}
\bigskip

\end{lemma}
\begin{proof}  In the four tables below
Lemma \ref{lem:AAAB}, evaluate
the parenthetical expressions using
Lemma \ref{lem:RR} along with Lemmas
\ref{lem:NTL1}, 
\ref{lem:NTL2}, 
\ref{lem:NTL3}.
\end{proof}

%%%%%%%%%%%%%%%%%%%%%%%%%

\noindent Recall the elements
$\lbrace S_i \rbrace_{i \in \mathbb Z_4}$
in $\widetilde \square_q$, from Definition
\ref{def:Deltai}.

\begin{proposition}
\label{prop:tildeMain}
For the algebra $\widetilde \square_q$ define
\begin{eqnarray*}
A = x_0+x_1, \qquad \qquad B = x_2+x_3.
\end{eqnarray*}
Then both
\begin{eqnarray}
&&
A^3 B - \lbrack 3 \rbrack_q A^2 B A +
\lbrack 3 \rbrack_q A B A^2 -B A^3 + (q^2-q^{-2})^2 c_0(AB-BA)
 = S_0 + S_1, \qquad 
\label{eq:1Assertion}
\\
&&
B^3 A - \lbrack 3 \rbrack_q B^2 A B +
\lbrack 3 \rbrack_q B A B^2 -A B^3 + (q^2-q^{-2})^2 c_2 (BA-AB)
= 
S_2 + S_3.
\label{eq:2Assertion}
\end{eqnarray}
\end{proposition}
\begin{proof}
To obtain 
(\ref{eq:1Assertion}),
write 
(\ref{eq:terms}) in the basis for
$\widetilde \square_q$ from
Proposition
\ref{prop:SquareBasis}(iv).
To do this, evaluate the terms
(\ref{eq:fiveTerms})
using
Lemma
\ref{lem:expand}, 
and the term $c_0(AB-BA)$ using
Lemma
\ref{lem:AAAB}(ii),(iii).
Assertion
(\ref{eq:1Assertion})
follows after a routine computation.
Assertion
(\ref{eq:2Assertion})
is obtained from
(\ref{eq:1Assertion})
by applying the square of
the automorphism
$\widetilde \rho$ from Lemma
\ref{lem:aut2}.
\end{proof}

\noindent {\it Proof of Proposition 
\ref{prop:ABxyiLongerM}}.
The composition
\begin{equation*}
\begin{CD} 
\widetilde \square_q  @>> \widetilde g(a,a^{-1},b,b^{-1}) >  
\widetilde \square_q  
 @>> can > \square_q
                  \end{CD}
\end{equation*}
is an $\mathbb F$-algebra homomorphism. 
Apply this homomorphism
to everything in Proposition
\ref{prop:tildeMain}.
$\hfill \square$

%%%%%%%%%%%%%%%%%%%%%%%%%%%
\section{The algebra $\widehat \square_q$}

\noindent We have been discussing the
algebras
$\square_q$ and
 $\widetilde \square_q$.
As we compare these algebras,
it is useful to bring in
 an ``intermediate'' algebra
$\widehat \square_q$
such that
 the canonical homomorphism
$ \widetilde \square_q \to \square_q$
has a factorization
$ \widetilde \square_q \to \widehat \square_q \to \square_q $.

\begin{definition}
\rm
\label{def:boxqV1}
Let $\widehat \square_q$ denote the $\mathbb F$-algebra with
generators $c^{\pm 1}_i, x_i$ $(i \in \mathbb Z_4)$
and relations
\begin{eqnarray}
&& \quad \qquad \qquad c_ic^{-1}_i = c^{-1}_i c_i = 1,
\label{eq:Box1}
\\
&& \quad \qquad \qquad 
\mbox{$c^{\pm 1}_i$ are central in $\widehat \square_q$},
\label{eq:Box2}
\\
&&
\quad \qquad \qquad
\frac{q x_i x_{i+1}-q^{-1}x_{i+1}x_i}{q-q^{-1}} = c_i,
\label{eq:central}
\\
&& 
x^3_i x_{i+2} - \lbrack 3 \rbrack_q x^2_i x_{i+2} x_i +
\lbrack 3 \rbrack_q x_i x_{i+2} x^2_i -x_{i+2} x^3_i = 0.
\label{eq:Box4}
\end{eqnarray}
\end{definition}

\noindent We have some comments.

\begin{lemma}
\label{lem:aut1}
There exists an automorphism $\widehat \rho$ of 
$\widehat \square_q$
that sends $c_i \mapsto c_{i+1}$ and $x_i \mapsto x_{i+1}$ for
$i \in \mathbb Z_4$. Moreover $\widehat \rho^4=1$.
\end{lemma}

\begin{lemma}
\label{lem:can01}
There exists an $\mathbb F$-algebra
homomorphism $\widetilde \square_q \to \widehat \square_q$
that sends $c^{\pm 1}_i \mapsto c^{\pm 1}_i$
and $x_i \mapsto x_i$ for
$i \in \mathbb Z_4$.
This homomorphism is surjective.
\end{lemma}

\begin{lemma}
\label{lem:can12}
There exists an $\mathbb F$-algebra
homomorphism 
$\widehat \square_q \to  \square_q$
that sends $c^{\pm 1}_i \mapsto 1$
and $x_i \mapsto x_i$ for
$i \in \mathbb Z_4$.
This homomorphism is surjective.
\end{lemma}

\begin{definition}
\label{def:cancan}
\rm The homomorphisms
 $\widetilde \square_q \to \widehat \square_q$
from Lemma
\ref{lem:can01}
and 
 $\widehat \square_q \to \square_q$
from Lemma
\ref{lem:can12} will be called {\it canonical}.
\end{definition}

\begin{note}
\rm In Definition
\ref{def:Canon}
we defined the canonical
homomorphism
$\widetilde \square_q \to \square_q$,
and in 
Definition
\ref{def:cancan} we defined the canonical homomorphisms
$\widetilde \square_q \to \widehat \square_q$ and
$\widehat \square_q \to \square_q$.
When we speak of the canonical homomorphism, it
should be clear from the context
which version we refer to.
\end{note}

\begin{lemma}
The following diagram commutes:

\begin{equation*}
\begin{CD}
\widetilde \square_q @>can > >
                  \widehat \square_q 
           \\ 
          @VIVV                     @VVcan V \\
            \widetilde \square_q @>>can > 
               \square_q 
                   \end{CD}
\end{equation*}

\end{lemma}
\begin{proof} By Definitions
\ref{def:Canon},
\ref{def:cancan}.
\end{proof}

\begin{definition}
\label{def:squareTeTo}
\rm Define the subalgebras 
$ \widehat \square^{\rm even}_q$,
$\widehat \square^{\rm odd}_q$,
$\widehat C$ of 
$ \widehat  \square_q$ such that
\begin{enumerate}
\item[\rm (i)]
$ \widehat \square^{\rm even}_q$
 is 
generated by $x_0, x_2$;
\item[\rm (ii)]
$ \widehat \square^{\rm odd}_q$
  is
generated by $x_1, x_3$;
\item[\rm (iii)]
  $\widehat C$ is
generated by $\lbrace c^{\pm 1}_i\rbrace_{i \in \mathbb Z_4}$.
\end{enumerate}
\end{definition}

\begin{lemma}
\label{lem:xiadj}
Let $\lbrace \alpha_i \rbrace_{i\in \mathbb Z_4}$
denote invertible elements in 
$\widehat C$.
Then there exists
an  $\mathbb F$-algebra homomorphism
$\widehat \square_q\to \widehat \square_q $ that sends
$x_i \mapsto \alpha_i x_i $ and
$c_i \mapsto \alpha_i \alpha_{i+1} c_i$ for $i \in \mathbb Z_4$.
\end{lemma}

\begin{definition}
\label{def:alphaAut}
\rm
The homomorphism in Lemma
\ref{lem:xiadj} will be denoted by $\widehat g(\alpha_0, \alpha_1, \alpha_2, \alpha_3)$.
\end{definition}

\begin{lemma} 
\label{lem:alphaAut}
Referring to Lemma
\ref{lem:xiadj} and Definition
\ref{def:alphaAut}, assume that $0 \not=\alpha_i \in \mathbb F$
for $i \in \mathbb Z_4$. Then
$\widehat g(\alpha_0, \alpha_1, \alpha_2, \alpha_3)$ is an automorphism
of $\widehat \square_q$. Its inverse is 
$\widehat g(\alpha^{-1}_0, \alpha^{-1}_1, \alpha^{-1}_2, \alpha^{-1}_3)$.
\end{lemma}
\begin{proof}
Similar to the proof of Lemma
\ref{lem:alphaAutT}.
\end{proof}

\noindent Our next goal is to obtain an analog of
Propositions \ref{thm:tensorDecPreM},
\ref{prop:tensorDec}
that applies to $\widehat \square_q$.

\begin{definition}  
\label{def:J}
\rm
 Let $J$ denote the 2-sided ideal of $\widetilde \square_q$
generated by $\lbrace S_i \rbrace_{i \in \mathbb Z_4}$.
Thus 
\begin{eqnarray}
\label{eq:Jdef}
J = \sum_{i \in \mathbb Z_4} 
\widetilde \square_q
S_i 
\widetilde \square_q.
\end{eqnarray}
\end{definition}

\begin{lemma}
\label{lem:canKer}
The canonical homomorphism 
$ \widetilde \square_q \to \widehat \square_q $ has
kernel $J$.
\end{lemma}
\begin{proof} Compare Definitions
\ref{def:boxqV2},
\ref{def:boxqV1}.
\end{proof}

\begin{definition}
\label{def:Jevenodd}
\rm
%%So the $\mathbb F$-algebra
%%$\square_q$ is isomorphic to
%%$\widetilde \square_q /J$.
 Let $J^{\rm even}$ 
(resp. 
$J^{\rm odd}$)
 denote the 2-sided ideal of
$\widetilde \square^{\rm even}_q$
(resp. 
$\widetilde \square^{\rm odd}_q$)
generated by
$S_0, S_2$
(resp. 
$S_1, S_3$).
Thus
\begin{eqnarray}
\label{eq:Jeve}
&&
J^{\rm even} =
\widetilde \square^{\rm even}_q 
S_0
\widetilde \square^{\rm even}_q 
+
\widetilde \square^{\rm even}_q 
S_2
\widetilde \square^{\rm even}_q,
\\
&&
\label{eq:Jodd}
J^{\rm odd} =
\widetilde \square^{\rm odd}_q 
S_1
\widetilde \square^{\rm odd}_q 
+
\widetilde \square^{\rm odd}_q 
S_3
\widetilde \square^{\rm odd}_q.
\end{eqnarray}
%%So the $\mathbb F$-algebra $\mathcal A_q$ is
%%isomorphic to $\widetilde \square^{\rm odd}_q/H^{\rm odd}$.
\end{definition}

\noindent Recall the free algebra $T$ with standard generators $x,y$.
Below 
(\ref{eq:SgenM}) 
we defined the 2-sided ideal $S$ of $T$.

\begin{lemma} 
\label{lem:imageIso}
The following {\rm (i), (ii)} hold:
\begin{enumerate}
\item[\rm (i)] $J^{\rm even}$ is the image of $S$
under the isomorphism 
$T \to \widetilde \square^{\rm even}_q$
from Proposition \ref{prop:tensorDec}(i);
\item[\rm (ii)] $J^{\rm odd}$ is the image of $S$
under the isomorphism 
$T \to \widetilde \square^{\rm odd}_q$
from
Proposition \ref{prop:tensorDec}(ii).
\end{enumerate}
\end{lemma}
\begin{proof}
Compare
Definition \ref{def:Jevenodd} with the definition of
$S$ below
(\ref{eq:SgenM}).
\end{proof}

\begin{lemma}
\label{lem:Jdesc}
Referring to the vector space isomorphism
from Proposition
\ref{prop:tensorDec}(iv), the preimage of
$J$ is
\begin{eqnarray}
\label{eq:Preim}
J^{\rm even} \otimes 
\widetilde \square^{\rm odd}_q \otimes \widetilde C + 
\widetilde \square^{\rm even}_q \otimes 
J^{\rm odd} \otimes \widetilde C.
\end{eqnarray}
\end{lemma}
\begin{proof} Let $\widetilde m$ denote the isomorphism in question.
Let $J'$ denote the image of
(\ref{eq:Preim}) under $\widetilde m$.
We show that $J=J'$. We have
$J \supseteq J'$ by construction and since
each of $J^{\rm even}, 
J^{\rm odd}$ 
is contained in $J$.
To obtain 
$J \subseteq J'$,  by
(\ref{eq:Jdef})
it suffices to show
that $\widetilde \square_q S_i \widetilde \square_q \subseteq J'$
for $i \in \mathbb Z_4$. Let $i$ be given,
and first assume that $i$ is even.
By Lemma
 \ref{lem:factor},
Corollary
\ref{cor:DeltaCom}, and since $\widetilde C$ is central in $\widetilde \square_q$,
\begin{eqnarray*}
\widetilde \square_q S_i \widetilde \square_q
=
\widetilde \square^{\rm even}_q
\widetilde \square^{\rm odd}_q
\widetilde C
S_i \widetilde \square_q
=
\widetilde \square^{\rm even}_q
S_i 
\widetilde \square^{\rm odd}_q
\widetilde C
\widetilde \square_q
\subseteq 
\widetilde \square^{\rm even}_q
S_i 
\widetilde \square_q.
\end{eqnarray*}
By Lemma
 \ref{lem:factor},
line (\ref{eq:Jeve}), and the definition of
$\widetilde m$,
\begin{eqnarray*}
\widetilde \square^{\rm even}_q
S_i 
\widetilde \square_q =
\widetilde \square^{\rm even}_q
S_i 
\widetilde \square^{\rm even}_q
\widetilde \square^{\rm odd}_q
\widetilde C
\subseteq 
J^{\rm even}
\widetilde \square^{\rm odd}_q
\widetilde C
=
\widetilde m(
J^{\rm even}
\otimes \widetilde \square^{\rm odd}_q
\otimes \widetilde C)
\subseteq J'.
\end{eqnarray*}
 We have shown
that $\widetilde \square_q S_i \widetilde \square_q \subseteq J'$
 for $i$ even.
We similarly show that
 $\widetilde \square_q S_i \widetilde \square_q \subseteq  J'$ for $i$ odd.
Therefore $J\subseteq J'$. We have shown that $J=J'$,
and the result follows.
\end{proof}

\begin{lemma}
\label{lem:Jint}
In the algebra $\widetilde \square_q$,
\begin{enumerate}
\item[\rm (i)]
 $
 J \cap 
\widetilde \square^{\rm even}_q
=
J^{\rm even}$;
\item[\rm (ii)]
$
J \cap \widetilde \square^{\rm odd}_q
=
J^{\rm odd}$;
\item[\rm (iii)]
$
 J \cap 
\widetilde C = 0$.
\end{enumerate}
\end{lemma}
\begin{proof}
By Lemma
\ref{lem:Jdesc} and since
neither of
$J^{\rm even},
J^{\rm odd}$ contains $1$.
\end{proof}

\noindent 
Recall from
 Definition
\ref{def:squareTeTo} 
the subalgebras
$\widehat \square^{\rm even}_q$,
$\widehat \square^{\rm odd}_q$,
$\widehat C$ of $\widehat \square_q$.

\begin{lemma} 
\label{lem:canonImage}
For the canonical homomorphism 
$ \widetilde \square_q \to \widehat \square_q $,
the images of
$\widetilde  \square^{\rm even}_q$,
$\widetilde  \square^{\rm odd}_q$,
$\widetilde  C$ are
$ \widehat \square^{\rm even}_q$
$ \widehat \square^{\rm odd}_q$,
$\widehat C$, respectively.
\end{lemma}
\begin{proof} By
Definitions
\ref{def:TeTo},
\ref{def:cancan},
\ref{def:squareTeTo}
and 
Lemma
\ref{lem:can01}.
\end{proof}

\begin{definition}
\rm
For the canonical homomorphism 
$ \widetilde \square_q \to \widehat \square_q$,
the restrictions to 
$\widetilde \square^{\rm even}_q$,
$\widetilde \square^{\rm odd}_q$,
$\widetilde C$
induce surjective
$\mathbb F$-algebra homomorphisms
\begin{eqnarray}
\widetilde \square^{\rm even}_q \to \widehat \square^{\rm even}_q,
\qquad \qquad
\widetilde \square^{\rm odd}_q \to \widehat \square^{\rm odd}_q,
\qquad \qquad
\widetilde C \to \widehat C.
\label{eq:threeRes}
\end{eqnarray}
Each of the homomorphisms 
(\ref{eq:threeRes})
will be called 
{\it restricted canonical}.
\end{definition}

\begin{lemma}
\label{lem:KerRest}
The following {\rm (i)--(iii)} hold:
\begin{enumerate}
\item[\rm (i)] the restricted canonical homomorphism 
$\widetilde \square^{\rm even}_q \to \widehat \square^{\rm even}_q$
has kernel $J^{\rm even}$;
\item[\rm (ii)] the restricted canonical homomorphism
$\widetilde \square^{\rm odd}_q \to \widehat \square^{\rm odd}_q$
has kernel $J^{\rm odd}$;
\item[\rm (iii)] the restricted canonical homomorphism
$\widetilde C \to \widehat C$
is a bijection.
\end{enumerate}
\end{lemma}
\begin{proof} 
By Lemmas
\ref{lem:canKer},
\ref{lem:Jint}.
\end{proof}

\begin{proposition}
\label{thm:tensorDec}
The following {\rm (i)--(iv)} hold:
\begin{enumerate}
\item[\rm (i)] 
there exists an $\mathbb F$-algebra isomorphism
$U^+ \to  \widehat \square^{\rm even}_q$ that sends
$X\mapsto x_0$ and
$Y\mapsto x_2$;
\item[\rm (ii)] 
there exists an $\mathbb F$-algebra isomorphism
$U^+ \to  \widehat \square^{\rm odd}_q$ that sends
$X\mapsto x_1$ and
$Y\mapsto x_3$;
\item[\rm (iii)] 
there exists an 
$\mathbb F$-algebra isomorphism
$L\to \widehat C$ 
that sends $\lambda^{\pm 1}_i \mapsto c^{\pm 1}_i$ for
$i \in \mathbb Z_4$;
\item[\rm (iv)] 
the following is an isomorphism of
$\mathbb F$-vector spaces:
\begin{eqnarray*}
\widehat \square^{\rm even}_q
\otimes
\widehat \square^{\rm odd}_q
\otimes
\widehat C
 & \to &  \widehat \square_q
\\
 u \otimes v \otimes c  &\mapsto & uvc
 \end{eqnarray*}
\end{enumerate}
\end{proposition}
\begin{proof} (i) 
Recall 
the free algebra $T$ with
standard generators $x,y$.
By Proposition
\ref{prop:tensorDec}(i) there exists an
$\mathbb F$-algebra isomorphism
$T \to \widetilde \square^{\rm even}_q$ 
that sends $x \mapsto x_0$ and
$y \mapsto x_2$.
The inverse of this isomorphism will be denoted by
$\widetilde \theta$.
Consider the
$\mathbb F$-algebra homomorphism
$\mu: T \to  U^+ $
that sends $x \mapsto X$ and
$y \mapsto Y$. By Lemma
\ref{lem:imageIso}(i),
the composition
\begin{equation*}
\begin{CD} 
\widetilde \square^{\rm even}_q  @>>  \widetilde \theta >  
 T @>> \mu > U^+
                  \end{CD}
\end{equation*}
is surjective with kernel $J^{\rm even}$.
Therefore there exists an $\mathbb F$-algebra isomorphism
$
\theta:
 \widetilde \square^{\rm even}_q /J^{\rm even} 
\to
U^+$
that sends
$ x_0 + J^{\rm even} \mapsto X$ and
$ x_2 + J^{\rm even} \mapsto Y$.
The restricted canonical homomorphism 
 $\widetilde \square^{\rm even}_q \to \widehat \square^{\rm even}_q$
sends $x_0 \mapsto x_0$ and
$x_2 \mapsto x_2$. This map is surjective by
construction, and has kernel $J^{\rm even}$
by Lemma
\ref{lem:KerRest}(i).
Therefore there exists an $\mathbb F$-algebra isomorphism
 $\vartheta :
 \widetilde \square^{\rm even}_q/J^{\rm even} \to \widehat \square^{\rm even}_q$
 that sends 
 $x_0+J^{\rm even} \mapsto x_0$ and
 $x_2+J^{\rm even} \mapsto x_2$.
The composition
\begin{equation*}
\begin{CD} 
U^+  @>>\theta^{-1} >
\widetilde \square^{\rm even}_q / J^{\rm even}   
  @>> \vartheta > \widehat \square^{\rm even}_q
                  \end{CD}
\end{equation*}
is the desired  $\mathbb F$-algebra isomorphism. 
\\
\noindent (ii), (iii). Similar to the proof of (i) above.
\\
\noindent (iv).
The multiplication map
$\widehat m:
 \widehat \square^{\rm even}_q
\otimes
\widehat \square^{\rm odd}_q
\otimes
\widehat C
  \to  \widehat \square_q
  $,
 $u \otimes v \otimes c  \mapsto  uvc$
 is $\mathbb F$-linear.
We show that $\widehat m$ is a bijection.
By Proposition
\ref{prop:tensorDec}(iv), the multiplication map
$\widetilde m:
 \widetilde \square^{\rm even}_q
\otimes
 \widetilde \square^{\rm odd}_q
\otimes
 \widetilde C
  \to \widetilde  \square_q
  $,
 $u \otimes v \otimes c  \mapsto  uvc$
 is an isomorphism of $\mathbb F$-vector spaces.
Recall the canonical homomorphism 
$\widetilde \square_q \to \widehat \square_q $ from
Definition
\ref{def:cancan}.
By construction the following diagram commutes:
\begin{equation*}
\begin{CD}
\mbox{
$
\widetilde \square^{\rm even}_q \otimes
\widetilde \square^{\rm odd}_q \otimes
\widetilde C
$  } @>
u \otimes v \otimes c \mapsto 
can(u)\otimes can(v) \otimes can(c)
>>
                \mbox{ $
 \widehat \square^{\rm even}_q \otimes
 \widehat \square^{\rm odd}_q \otimes
 \widehat C$
		} 
           \\ 
          @V\widetilde m VV                     @VV \widehat m V \\
                \mbox{$\widetilde \square_q $} @>>{can}> 
                \mbox{$ \widehat \square_q $}
                   \end{CD}
\end{equation*}
The map $\widehat m$ is
surjective by these comments and
the last assertion of Lemma
\ref{lem:can01}.
%%%%%%%%\ref{lem:squareCanon}.
The map $\widehat m$
is
injective 
in view of
Lemma
\ref{lem:Jdesc}
along with Lemmas
\ref{lem:canKer},
\ref{lem:KerRest}.
 Therefore $\widehat m$ is a bijection.
\end{proof}

%%%%%%%%%%\section{How $\widehat \square_q$ is related to $\mathcal O$}

\noindent For the rest of this section, we
 describe how $\widehat \square_q$ is related to the $q$-Onsager algebra
 $\mathcal O$.

\begin{proposition}
\label{prop:ABxy}
For the algebra $\widehat \square_q$, define
\begin{eqnarray*}
A = x_0+x_1, \qquad \qquad B = x_2+x_3.
\end{eqnarray*}
Then
\begin{eqnarray*}
&&
A^3 B - \lbrack 3 \rbrack_q A^2 B A +
\lbrack 3 \rbrack_q A B A^2 -B A^3 = (q^2-q^{-2})^2 c_0(BA-AB),
\\
&&
B^3 A - \lbrack 3 \rbrack_q B^2 A B +
\lbrack 3 \rbrack_q B A B^2 -A B^3 = (q^2-q^{-2})^2 c_2(AB-BA).
\end{eqnarray*}
\end{proposition}
\begin{proof}
Apply the canonical homomorphism $\widetilde \square_q \to \widehat \square_q$
to each side of
(\ref{eq:1Assertion}),
(\ref{eq:2Assertion}).
\end{proof}

\noindent The following more general
version of 
 Proposition
\ref{prop:ABxy} is obtained by applying
Lemma
\ref{lem:xiadj}.

\begin{corollary}
\label{prop:ABxyiLong}
 Let
 $\lbrace \alpha_i \rbrace_{i\in \mathbb Z_4}$ denote invertible
elements in 
$\widehat C$.
For the algebra
$\widehat \square_q$,
define
\begin{eqnarray}
\label{eq:ABDef}
A = \alpha_0 x_0+ \alpha_1 x_1, \qquad \qquad B = \alpha_2 x_2+ \alpha_3 x_3.
\end{eqnarray}
Then
\begin{eqnarray}
&&
A^3 B - \lbrack 3 \rbrack_q A^2 B A +
\lbrack 3 \rbrack_q A B A^2 -B A^3 =
 (q^2-q^{-2})^2 \alpha_0 \alpha_1 c_0(BA-AB),
\label{eq:qDGLong1}
\\
&&
B^3 A - \lbrack 3 \rbrack_q B^2 A B +
\lbrack 3 \rbrack_q B A B^2 -A B^3 = 
(q^2-q^{-2})^2 \alpha_2 \alpha_3 c_2(AB-BA).
\label{eq:qDGLong2}
\end{eqnarray}
\end{corollary}
\begin{proof}
Apply the homomorphism
$\widehat g(\alpha_0,\alpha_1,\alpha_2,\alpha_3)$
from Definition
\ref{def:alphaAut}
to everything in 
Proposition
\ref{prop:ABxy}.
\end{proof}

\begin{proposition}
\label{prop:ABxyiLonger}
 Let
 $\lbrace \alpha_i \rbrace_{i\in \mathbb Z_4}$ denote
  elements in 
$\widehat C$ such that 
 $\alpha_0 \alpha_1 c_0=1 $ and
 $\alpha_2 \alpha_3 c_2=1$.
Then there exists an $\mathbb F$-algebra 
homomorphism
$ \natural: \mathcal O \to \widehat \square_q$ that sends
\begin{eqnarray}
\label{eq:ABDefMap}
A \mapsto \alpha_0 x_0+ \alpha_1 x_1,
\qquad \qquad 
B \mapsto \alpha_2 x_2+ \alpha_3 x_3.
\end{eqnarray}
\end{proposition}
\begin{proof}
By Definition
\ref{def:qOnsager}
and
Corollary
\ref{prop:ABxyiLong}.
\end{proof}

\noindent 
In 
Theorem \ref{thm:naturalINJ}
we will show that
the map $\natural$ from
Proposition
\ref{prop:ABxyiLonger} is injective.
\medskip

\noindent
Let the elements
 $\lbrace \alpha_i \rbrace_{i\in \mathbb Z_4}$ 
and the map $\natural$ be as in
 Proposition
\ref{prop:ABxyiLonger}.  By construction
there exist nonzero $a,b \in \mathbb F$
such that
the canonical homomorphism $\widehat \square_q \to \square_q$
sends
\begin{eqnarray}
\alpha_0 \mapsto a, \qquad  \quad
\alpha_1 \mapsto a^{-1}, \qquad  \quad
\alpha_2 \mapsto b, \qquad  \quad
\alpha_3 \mapsto b^{-1}.
\label{eq:canonsend}
\end{eqnarray}
Using $a,b$ we obtain the
$\mathbb F$-algebra homomorphism
$\sharp: \mathcal O \to \square_q$
as in
Proposition
\ref{prop:ABxyiLongerM}.

\begin{lemma}
\label{lem:naturalSharp}
With the above notation,
the following diagram commutes:

\begin{equation*}
\begin{CD}
\mathcal O @>\natural > >
                 \widehat \square_q 
           \\ 
          @VIVV                     @VV can V \\
          \mathcal O       @>>\sharp> 
               \square_q 
                   \end{CD}
\end{equation*}
\end{lemma}
\begin{proof} 
Each map in the diagram is an $\mathbb F$-algebra
homomorphism.
The $\mathbb F$-algebra $\mathcal O$
is generated by $A,B$. To verify that the diagram commutes,
chase $A,B$ around the diagram
using
(\ref{eq:ABDefMapM}) and
(\ref{eq:ABDefMap}),
(\ref{eq:canonsend}).
\end{proof}

\section{The algebra $\square_q$, revisited}

\noindent
In Section 5 we  defined the algebra $\square_q$, and
in Sections 6--9 we investigated its homomorphic
preimages
 $\widetilde \square_q$,
$\widehat \square_q$.
In this section we return our attention to $\square_q$.
We will first prove Proposition
\ref{thm:tensorDecPreM}. We will then prove
Theorem
\ref{thm:main1}, and establish the
injectivity of the 
 maps $\sharp$ and $\natural$ from
Propositions
\ref{prop:ABxyiLongerM} and
\ref{prop:ABxyiLonger}, respectively.
\medskip

\noindent Recall the $\mathbb F$-algebra $L$ from
(\ref{eq:Ldef}). Let $L_0$ denote the ideal of $L$ generated by
$\lbrace \lambda_i -1 \rbrace_{i\in \mathbb Z_4}$.
Thus
\begin{eqnarray}
L_0 = \sum_{i \in \mathbb Z_4} L(\lambda_i-1).
\label{eq:L0}
\end{eqnarray}
The ideal $L_0$ is the kernel of the 
$\mathbb F$-algebra
homomorphism $L\to \mathbb F$ that sends $\lambda_i \mapsto 1$
for $i \in \mathbb Z_4$. 
The sum $L = L_0 + \mathbb F$ is direct.

\begin{definition}\rm
\label{def:K}
Let $K$ denote the 2-sided ideal of $\widehat \square_q$ generated by
$\lbrace c_i -1 \rbrace_{i \in \mathbb Z_4}$. Since the
$\lbrace c_i \rbrace_{i \in \mathbb Z_4}$ are central,
\begin{eqnarray*}
K = \sum_{i \in \mathbb Z_4} \widehat \square_q (c_i-1).
\end{eqnarray*}
\end{definition}
\begin{lemma}
\label{lem:Kmeaning}
The canonical homomorphism $\widehat \square_q \to \square_q$
has kernel $K$.
\end{lemma}
\begin{proof} Compare Definition
\ref{def:boxqV1M}
and Definition
\ref{def:boxqV1}.
\end{proof}

\noindent Recall from Definition 
\ref{def:squareTeTo}
the subalgebras 
$ \widehat \square^{\rm even}_q$,
$\widehat \square^{\rm odd}_q$,
$\widehat C$ of 
$ \widehat  \square_q$.

\begin{definition}\rm
\label{def:K0}
Let $K_0$ denote the ideal of $\widehat C$ generated by
$\lbrace c_i -1\rbrace_{i \in \mathbb Z_4}$. Thus
\begin{eqnarray}
\label{eq:K0}
K_0 = \sum_{i \in \mathbb Z_4} \widehat C (c_i-1).
\end{eqnarray}
\end{definition}

\begin{lemma}
The ideal $K_0 $ is the image of $L_0$ under the isomorphism
$L \to \widehat C$ from Proposition
\ref{thm:tensorDec}(iii).
\end{lemma}
\begin{proof} Compare
(\ref{eq:L0}),
(\ref{eq:K0}).
\end{proof}

\begin{lemma}
\label{lem:Kdesc}
Referring to the vector space isomorphism
from Proposition
\ref{thm:tensorDec}(iv), the preimage of $K$ is 
\begin{eqnarray}
\label{eq:SSK}
\widehat \square^{\rm even}_q
\otimes
\widehat \square^{\rm odd}_q
\otimes
K_0.
\end{eqnarray}
\end{lemma}
\begin{proof} By 
Definitions \ref{def:K},
\ref{def:K0}.
%%%%%%%%and since
%%%%the $\lbrace c_i \rbrace_{i \in \mathbb Z_4}$ are central in
%%%%$\widehat \square_q$.
\end{proof}

\begin{lemma} 
\label{lem:Kintersect}
In the algebra $\widehat \square_q$,
\begin{enumerate}
\item[\rm (i)] $K \cap \widehat \square^{\rm even}_q = 0$;
\item[\rm (ii)] $K \cap \widehat \square^{\rm odd}_q = 0$;
\item[\rm (iii)] $K \cap \widehat C= K_0$.
\end{enumerate}
\end{lemma}
\begin{proof} Use Lemma
\ref{lem:Kdesc}.
\end{proof}

\noindent 
Recall from
 Definition
\ref{def:squareTeToM}
the subalgebras
$\square^{\rm even}_q$,
$\square^{\rm odd}_q$
of 
$\square_q$.

\begin{lemma} 
\label{lem:canonImageA}
For the canonical homomorphism 
$ \widehat \square_q \to  \square_q $,
the images of
$\widehat  \square^{\rm even}_q$,
$\widehat  \square^{\rm odd}_q$,
$\widehat  C$
are
$ \square^{\rm even}_q$
$ \square^{\rm odd}_q$,
$\mathbb F$,
respectively.
\end{lemma}
\begin{proof} By
Definitions
\ref{def:squareTeToM},
\ref{def:squareTeTo}
and Lemma
\ref{lem:can12}.
\end{proof}

\begin{definition}
\rm
For the canonical homomorphism 
$ \widehat \square_q \to \square_q$,
the restrictions to 
$\widehat \square^{\rm even}_q$,
$\widehat \square^{\rm odd}_q$,
$\widehat C$
induce surjective
$\mathbb F$-algebra homomorphisms
\begin{eqnarray}
\widehat \square^{\rm even}_q \to  \square^{\rm even}_q,
\qquad \qquad
\widehat \square^{\rm odd}_q \to  \square^{\rm odd}_q,
\qquad \qquad
\widehat C \to  \mathbb F.
\label{eq:threeResA}
\end{eqnarray}
Each of the homomorphisms 
(\ref{eq:threeResA})
will be called 
{\it restricted canonical}.
\end{definition}

\begin{lemma}
\label{lem:KerRestA}
The following {\rm (i)--(iii)} hold:
\begin{enumerate}
\item[\rm (i)] the restricted canonical homomorphism 
$\widehat \square^{\rm even}_q \to \square^{\rm even}_q$
is an isomorphism;
\item[\rm (ii)] the restricted canonical homomorphism
$\widehat \square^{\rm odd}_q \to \square^{\rm odd}_q$
is an isomorphism;
\item[\rm (iii)] the restricted canonical homomorphism
$\widehat C \to \mathbb F$
has kernel $K_0$.
\end{enumerate}
\end{lemma}
\begin{proof}  By Lemmas
\ref{lem:Kmeaning},
\ref{lem:Kintersect}.
\end{proof}

\noindent {\it Proof of Proposition 
\ref{thm:tensorDecPreM}}.
(i) 
The desired isomorphism is the composition
of the isomorphism $U^+ \to \widehat \square^{\rm even}_q$
from Proposition
\ref{thm:tensorDec}(i)
and the isomorphism
$\widehat \square^{\rm even}_q \to \square^{\rm even}_q$
from
Lemma \ref{lem:KerRestA}(i).
\\
\noindent (ii) 
The desired isomorphism is the composition
of the isomorphism $U^+ \to \widehat \square^{\rm odd}_q$
from Proposition
\ref{thm:tensorDec}(ii)
and the isomorphism
$\widehat \square^{\rm odd}_q \to \square^{\rm odd}_q$
from
Lemma \ref{lem:KerRestA}(ii).
\\
\noindent (iii) By Lemma
\ref{lem:Kdesc} and since $1 \not\in K_0$ we see that, for the 
vector space isomorphism
in Proposition
\ref{thm:tensorDec}(iv),
the preimage of $K$ has zero  intersection with
$\widehat \square^{\rm even}_q
\otimes
\widehat \square^{\rm odd}_q$.
The result follows.
$\hfill \square$
\medskip

\noindent We now turn our attention to 
Theorem
\ref{thm:main1} and the maps
$\sharp$, $\natural$.
We comment on the notation.
In earlier sections we discussed the
algebras
$\square_q$,
$\square^{\rm even}_q$,
$\square^{\rm odd}_q$. 
In our discussion going forward, in order to simplify
the notation 
we will
drop the reference to $q$, and write
$\square=\square_q$,
$\square^{\rm even}=\square^{\rm even}_q$,
$\square^{\rm odd}=\square^{\rm odd}_q$.
\medskip

\noindent
Recall the $\mathbb F$-algebra $U^+=U^+_q$ and
its $\mathbb N$-grading
$\lbrace U^+_n\rbrace_{n \in \mathbb N}$.

\begin{definition}
\label{def:recall1}
\rm
Recall from Proposition
\ref{thm:tensorDecPreM}(i)
the $\mathbb F$-algebra isomorphism
$U^+ \to \square^{\rm even}$ that sends $X \mapsto x_0$
and $Y \mapsto x_2$.
Under this isomorphism,
for $n \in \mathbb N$ the image of
$U^+_n$ will be denoted 
by $\square^{\rm even}_n$.
Recall from Proposition
\ref{thm:tensorDecPreM}(ii)
the $\mathbb F$-algebra isomorphism
$U^+ \to \square^{\rm odd}$ that sends $X \mapsto x_1$
and $Y \mapsto x_3$.
Under this isomorphism,
for $n \in \mathbb N$ the image of
$U^+_n$ will be denoted by
 $\square^{\rm odd}_n$.
 \end{definition}

\begin{lemma} The sequence
$\lbrace 
 \square^{\rm even}_n \rbrace_{n \in \mathbb N}$ is an $\mathbb N$-grading
 of 
 $\square^{\rm even}$.
The sequence
$\lbrace 
 \square^{\rm odd}_n \rbrace_{n \in \mathbb N}$ is an $\mathbb N$-grading
 of 
 $\square^{\rm odd}$.
\end{lemma}
\begin{proof}
By Definition
\ref{def:recall1} and since
$\lbrace 
 U^+_n \rbrace_{n \in \mathbb N}$ is an $\mathbb N$-grading
 of $U^+$.
 \end{proof}

\begin{definition}
\label{def:squareNPre}
\rm
Recall from Proposition
\ref{thm:tensorDecPreM}(iii)
the isomorphism of $\mathbb F$-vector spaces
$\square^{\rm even} \otimes
\square^{\rm odd}  \to \square$.
Under this isomorphism,
for $r,s \in \mathbb N$  the image of
$
\square^{\rm even}_r
\otimes
\square^{\rm odd}_s
$
will be denoted by $\square_{r,s}$.
\end{definition}

\noindent Referring to Definition
\ref{def:squareNPre},
the sum $\square= \sum_{r,s \in \mathbb N} \square_{r,s}$ is
direct.
Moreover $1 \in \square_{0,0}$ and
$x_0,x_2 \in \square_{1,0}$ and
$x_1,x_3 \in \square_{0,1}$.

\begin{lemma}
\label{lem:rstcom}
The following
{\rm (i), (ii)} hold for $r,s,t \in \mathbb N$:
\begin{enumerate}
\item[\rm (i)]
$\square^{\rm even}_t \square_{r,s} \subseteq \square_{r+t,s}$;
\item[\rm (ii)]
$\square_{r,s}\square^{\rm odd}_t \subseteq \square_{r,s+t}$.
\end{enumerate}
\end{lemma}
\begin{proof} (i) By Definition
\ref{def:squareNPre} and since
$\square^{\rm even}_t 
\square^{\rm even}_r \subseteq \square^{\rm even}_{r+t}$. 
\\
\noindent (ii) By Definition
\ref{def:squareNPre} and since
$\square^{\rm odd}_s 
\square^{\rm odd}_t \subseteq \square^{\rm odd}_{s+t}$. 
\end{proof}

\begin{definition}
\label{def:squareN}
\rm
For $n \in \mathbb Z$ 
define
\begin{eqnarray*}
\square_n = 
\sum_
{
\genfrac{}{}{0pt}{}{r,s \in \mathbb N}{r-s=n}
} 
\square_{r,s}.
\end{eqnarray*}
\end{definition}

\noindent 
Referring to Definition
\ref{def:squareN},
our next goal is to show that the sequence
$\lbrace \square_n\rbrace_{n \in \mathbb Z}$ is
a $\mathbb Z$-grading of $\square$.

\begin{lemma}
\label{lem:grade12a}
The sum 
$\square= \sum_{n\in \mathbb Z} \square_n$ is direct.
Moreover
 $ 1 \in  \square_0$ and
$x_0,x_2 \in \square_1$ and
$x_1,x_3 \in \square_{-1}$.
\end{lemma}
\begin{proof} By Definition
\ref{def:squareN}
and the comments
above Lemma
\ref{lem:rstcom}.
\end{proof}

\begin{lemma}
\label{lem:grade12}
The following {\rm (i), (ii)}
hold for $n\in \mathbb Z$ and $t\in \mathbb N$:
\begin{enumerate}
\item[\rm (i)] $\square^{\rm even}_t \square_n \subseteq
\square_{n+t}$;
\item[\rm (ii)] $\square_n \square^{\rm odd}_t \subseteq
\square_{n-t}$.
\end{enumerate}
\end{lemma}
\begin{proof}
Use Lemma
\ref{lem:rstcom}
and Definition 
\ref{def:squareN}.
\end{proof}

\begin{lemma}
\label{lem:grading1}
For $r,s\in \mathbb N$,
\begin{eqnarray}
\label{eq:outoforder}
\square^{\rm odd}_s \square^{\rm even}_r \subseteq 
\sum_{\ell=0}^{{\rm {min}} (r,s)}
\square_{r-\ell,s-\ell}.
\end{eqnarray}
Moreover
$\square^{\rm odd}_s 
\square^{\rm even}_r \subseteq \square_{r-s}$. 
\end{lemma}
\begin{proof} To obtain
(\ref{eq:outoforder}) 
use Lemma
\ref{lem:uLong}
and induction on $s$.
The last assertion follows from
(\ref{eq:outoforder})
and Definition 
\ref{def:squareN}.
\end{proof}

\begin{lemma}
\label{lem:grading2}
We have $\square_r \square_s \subseteq \square_{r+s}$
for 
 $r,s\in \mathbb Z$.
\end{lemma}
\begin{proof} By
Definition
\ref{def:squareN},
Lemma \ref{lem:grade12}, and the last assertion of
Lemma
\ref{lem:grading1}.
\end{proof}

\begin{proposition}
\label{prop:Grading}
The sequence $\lbrace \square_n\rbrace_{n \in \mathbb Z}$
is a $\mathbb Z$-grading of $\square$.
\end{proposition}
\begin{proof} By Lemmas
\ref{lem:grade12a},
\ref{lem:grading2}.
\end{proof}

\noindent 
For the rest of this section, fix nonzero $a,b\in \mathbb F$.
Using $a,b$ we obtain the $\mathbb F$-algebra homomorphism
$\sharp: \mathcal O \to \square$ as in
Proposition
\ref{prop:ABxyiLongerM}.
Define
\begin{eqnarray}
A^+ = a x_0, \qquad
A^- = a^{-1} x_1, \qquad
B^+ = b x_2, \qquad
B^- = b^{-1} x_3
\label{def:ABPM}
\end{eqnarray}
so that
\begin{eqnarray}
\sharp(A) = A^+ + A^-, \qquad \qquad 
\sharp(B) = B^+ + B^-. 
\label{eq:ABxi}
\end{eqnarray}
\noindent By Lemma
\ref{lem:grade12a} and
(\ref{def:ABPM}),
\begin{eqnarray}
A^+ \in \square_1, \qquad
A^- \in \square_{-1}, \qquad
B^+ \in \square_{1}, \qquad
B^- \in \square_{-1}.
\label{eq:ABloc}
\end{eqnarray}
By 
(\ref{eq:ABxi}),
(\ref{eq:ABloc}),
\begin{eqnarray}
\sharp (A) \in \square_1  + \square_{-1},
\qquad \qquad 
\sharp (B) \in \square_1  + \square_{-1}.
\label{lem:xiAB} 
\end{eqnarray}

\begin{lemma}
\label{lem:wordExpand}
Let $n \in \mathbb N$. For $1 \leq i \leq n$ pick
$g_i \in \lbrace A,B\rbrace$. Then 
\begin{eqnarray}
\label{eq:xiExpand}
\sharp(g_1g_2\cdots g_n) =
\sum
g^{\epsilon_1}_1 
g^{\epsilon_2}_2
\cdots
g^{\epsilon_n}_n,
\end{eqnarray}
where the sum is over all sequences
$\epsilon_1, 
\epsilon_2,
\ldots,
\epsilon_n$
such that $\epsilon_i \in \lbrace +,-\rbrace$ for
$1 \leq i \leq n$.
\end{lemma}
\begin{proof} To verify
(\ref{eq:xiExpand}),
expand
the left-hand side
using 
(\ref{eq:ABxi}) and the fact that $\sharp$
is an algebra homomorphism.
\end{proof}

\begin{lemma}
\label{lem:whereSummand}
Refer to Lemma
\ref{lem:wordExpand}. 
In the sum on the right in
{\rm (\ref{eq:xiExpand})}, consider any summand
$g^{\epsilon_1}_1 
g^{\epsilon_2}_2
\cdots
g^{\epsilon_n}_n$. This summand is contained in $\square_\ell$,
where
\begin{eqnarray*}
\ell
= 
|\lbrace i|1 \leq i \leq n,\; \epsilon_i = +\rbrace |
-
|\lbrace i|1 \leq i \leq n, \,\epsilon_i = -\rbrace |.
\end{eqnarray*}
\end{lemma}
\begin{proof}
By Lemma
\ref{lem:grading2}
and
(\ref{eq:ABloc}).
\end{proof}

\begin{lemma}
\label{lem:spread}
For $n \in \mathbb N$,
\begin{eqnarray}
\sharp(\mathcal O_n) \subseteq \sum_{\ell=-n}^n \square_{\ell}.
\label{eq:spread}
\end{eqnarray}
\end{lemma}
\begin{proof} 
We mentioned below Definition
\ref{def:qOnsager}
that
$\mathcal O_n$ is spanned by
the products
$g_1g_2\cdots g_r$ such that
$0 \leq r \leq n$ and
$g_i$ is among $A,B$ for $1 \leq i \leq r$.
The result follows from this and
Lemmas
\ref{lem:wordExpand},
\ref{lem:whereSummand}.
\end{proof}

\begin{definition}
\label{def:projGrad}
\rm
For $r\in \mathbb Z$ define an $\mathbb F$-linear
map $\pi_r: \square \to \square$ such that
$(\pi_r-I) \square_r = 0$ and
$\pi_r \square_s = 0$ for $s \in \mathbb Z$, $s\not=r$.
Thus $\pi_r$ is the projection from $\square$ onto
$\square_r$.
\end{definition}

\begin{lemma}
\label{lem:comp}
For $n \in \mathbb N$ and 
$r \in \mathbb Z$,
the composition
\begin{equation*}
\begin{CD} 
\mathcal O_n  @>>  \sharp >  
               \square  @>> \pi_r >  \square
                  \end{CD}
\end{equation*}
is zero unless $-n \leq r \leq n$.
\end{lemma}
\begin{proof}
By Lemma
\ref{lem:spread}
and Definition
\ref{def:projGrad}.
\end{proof}

\begin{definition}
\label{def:comp}
Pick $n \in \mathbb N$ and recall the $\mathbb F$-vector
space 
$\overline{\mathcal O}_n = \mathcal O_n /\mathcal O_{n-1}$
from Section 4.
Define the $\mathbb F$-linear map
$\varphi_n : \mathcal O_n \to \square$ to be 
the composition
\begin{equation}
\begin{CD} 
\varphi_n: \mathcal O_n  @>>  \sharp >  
               \square  @>> \pi_n >  \square.
                  \end{CD}
\label{eq:compVarphi}
\end{equation}
By Lemma
\ref{lem:comp}
we have $\varphi_n(\mathcal O_{n-1})=0$.
Therefore
$\varphi_n$ induces an $\mathbb F$-linear
map $\overline{\varphi}_n: 
\overline{\mathcal O}_n \to \square$ that sends
$u+\mathcal O_{n-1} \mapsto \varphi_n(u)$ for all 
$u \in \mathcal O_n$.
\end{definition}

\begin{lemma} 
\label{lem:basisAction}
For $n \in \mathbb N$ 
the map 
$\overline{\varphi}_n: \overline {\mathcal O}_n \to \square$
is described as follows.
For $1 \leq i \leq n$ pick $g_i \in \lbrace A,B\rbrace$.
Then $\overline{\varphi}_n$ sends
$\overline{g}_1 \overline{g}_2 \cdots \overline{g}_n \mapsto
g^+_1g^+_2 \cdots g^+_n$, where we recall $A^+= a x_0$
and $B^+ = b x_2$.
\end{lemma}
\begin{proof}
By construction 
$
\overline{g}_1 \overline{g}_2 \cdots \overline{g}_n
=
g_1 g_2 \cdots g_n+\mathcal O_{n-1} \in \overline{\mathcal O}_n$.
By this and Definition
\ref{def:comp},
the map 
$\overline{\varphi}_n$ sends
$\overline{g}_1 \overline{g}_2 \cdots \overline{g}_n$
to $\varphi_n(
g_1 g_2 \cdots g_n)$, which is
equal to
$\pi_n(\sharp(g_1 g_2 \cdots g_n))$.
To compute this last quantity, apply
$\pi_n$ to each side of
(\ref{eq:xiExpand}), and
evaluate the result using
Lemma
\ref{lem:whereSummand} and
Definition
\ref{def:projGrad}.
For the sum
on the right in
(\ref{eq:xiExpand}),
$\pi_n$ fixes
the summand 
$g^+_1g^+_2 \cdots g^+_n$ and sends the remaining
summands to zero.
The result follows.
\end{proof}

\begin{lemma}
\label{lem:rsvarphi}
Let $r,s \in \mathbb N$. Then
$\overline{\varphi}_r(u)
\overline{\varphi}_s(v)
=
\overline{\varphi}_{r+s}(uv)$
for
all $u\in \overline{\mathcal O}_r$ and
 $v\in \overline{\mathcal O}_s$.
\end{lemma}
\begin{proof} Use Lemma
\ref{lem:basisAction}.
\end{proof}

\begin{definition}
\label{def:varphi}
Define an $\mathbb F$-linear map
$\overline{\varphi} :\overline{\mathcal O} \to 
\square$ that acts on
$\overline{\mathcal O}_n$ as $\overline{\varphi}_n$
 for $n \in \mathbb N$.
\end{definition}

\begin{lemma} The map
$\overline{\varphi} :\overline{\mathcal O} \to 
\square$ from Definition
\ref{def:varphi} is an $\mathbb F$-algebra homomorphism.
\end{lemma}
\begin{proof} 
It suffices to check that 
$\overline{\varphi}(u)
\overline{\varphi}(v)=
\overline{\varphi}(uv)$ 
 for
all $u,v\in \overline{\mathcal O}$.
This checking is routine
using Lemma
\ref{lem:rsvarphi}, Definition
\ref{def:varphi}, and 
the definition of the $\mathbb F$-algebra
$\overline{\mathcal O}$ in Section 4.
\end{proof}

\begin{lemma}
\label{lem:barphiAction}
The map
$\overline{\varphi} :\overline{\mathcal O} \to 
\square$ from Definition
\ref{def:varphi} sends
$\overline{A} \mapsto ax_0$ and
$\overline{B} \mapsto bx_2$.
\end{lemma}
\begin{proof} We have
$\overline{A}, \overline{B}  \in \overline{\mathcal O}_1$.
By
Lemma
\ref{lem:basisAction} and
Definition
\ref{def:varphi},
\begin{eqnarray*}
\overline{\varphi}(\overline{A}) = 
\overline{\varphi}_1(\overline{A}) = A^+ = ax_0,
\qquad \qquad 
\overline{\varphi}(\overline{B}) = 
\overline{\varphi}_1(\overline{B}) = B^+ = bx_2.
\end{eqnarray*}
\end{proof}

\noindent
By 
Lemma
\ref{lem:xi} there exists an automorphism of
$U^+$ that sends $X \mapsto aX$ and
$Y \mapsto bY$. Denote this automorphism by $\phi$.
Recall the surjective $\mathbb F$-algebra homomorphism
$\psi: U^+ \to \overline{\mathcal O}$ from
above
Lemma
\ref{lem:munu}.
By Proposition
\ref{thm:tensorDecPreM}(i)
there exists an injective $\mathbb F$-algebra
homomorphism $U^+ \to \square$ that sends
$X \mapsto x_0$ and 
$Y \mapsto x_2$.
Call this homomorphism $\flat $.

\begin{lemma}
\label{lem:bigCom}
With the above notation, the following diagram commutes:

\begin{equation*}
\begin{CD}
U^+ @>\psi > >
                 \overline{\mathcal O} 
           \\ 
          @V \phi VV                     @VV\overline{\varphi} V \\
                U^+ @>>\flat> 
               \square 
                   \end{CD}
\end{equation*}

\end{lemma}
\begin{proof} Each map in the diagram
is an $\mathbb F$-algebra homomorphism.
The $\mathbb F$-algebra $U^+$ is generated by
$X,Y$.
To verify that the diagram commutes,
chase $X,Y$ around
the diagram using
Lemma
\ref{lem:barphiAction}.
\end{proof}

\begin{proposition}
\label{prop:twoInj}
Theorem
\ref{thm:main1}
holds. Moreover
the map
$\overline{\varphi}$ is injective.
\end{proposition}
\begin{proof} 
Consider the commuting diagram in Lemma
\ref{lem:bigCom}.
By construction $\phi $ is an isomorphism,
$\flat$ is injective,
and $\psi$ is surjective.
Now since the diagram commutes, 
$\psi$ and $\overline{\varphi}$ are injective.
The map $\psi$ is an isomorphism since
it is both injective and surjective.
Therefore Theorem
\ref{thm:main1}
holds.
\end{proof}

\begin{theorem}
\label{thm:xiInj}
The map $\sharp $ from
Proposition
\ref{prop:ABxyiLongerM} is injective.
\end{theorem}
\begin{proof}
Pick a nonzero $u \in \mathcal O$.
We show that
$\sharp (u)\not=0$.
There exists a unique $n \in \mathbb N$ such that $
u \in \mathcal O_n$ and
$ u \not\in \mathcal O_{n-1}$.
The map $\overline{\varphi}_n$ is injective,
by Definition
\ref{def:varphi} and since $\overline{\varphi}$ is injective by
Proposition
\ref{prop:twoInj}. By this and
Definition
\ref{def:comp}, the kernel of $\varphi_n$
is equal to  
 $\mathcal O_{n-1}$. This kernel does not
contain $u$, so 
$\varphi_n(u) \not=0$. Now $\sharp (u)\not=0$ in view of 
(\ref{eq:compVarphi}).
The result follows.
\end{proof}

\begin{theorem}
\label{thm:naturalINJ}
The map $\natural$ from Proposition
\ref{prop:ABxyiLonger}
is injective.
\end{theorem}
\begin{proof} By Lemma
\ref{lem:naturalSharp}
and Theorem 
\ref{thm:xiInj}.
\end{proof}

\section{The $q$-tetrahedron algebra $\boxtimes_q$ }

\noindent  The
$q$-tetrahedron algebra $\boxtimes_q$ was introduced in
\cite{qtet} and investigated further in
\cite{irt},
\cite{miki}.
In this section, we consider how 
$\mathcal O$ and $\square_q$ are related to $\boxtimes_q$.
We first 
display an injective $\mathbb F$-algebra
homomorphism $\square_q \to
\boxtimes_q$. Next, we compose this  map with
the map
$\sharp: \mathcal O \to \square_q$ from
Proposition
\ref{prop:ABxyiLongerM}, 
to get an injective $\mathbb F$-algebra
homomorphism
$\mathcal O \to \boxtimes_q$.

\begin{definition}
\label{def:tet}
\rm
(See \cite[Definition~6.1]{qtet}.)
Let $\boxtimes_q$ denote the $\F$-algebra defined by generators
\begin{eqnarray}
\lbrace x_{ij}\;|\;i,j \in
\mathbb Z_4,
\;\;j-i=1 \;\mbox{\rm {or}}\;
j-i=2\rbrace
\label{eq:gen}
\end{eqnarray}
and the following relations:
\begin{enumerate}
\item[\rm (i)] For $i,j \in 
\Z_4$ such that $j-i=2$,
\begin{eqnarray}
 x_{ij}x_{ji} =1.
 \label{eq:tet1}
 \end{eqnarray}
 \item[\rm (ii)] For $i,j,k \in 
 \Z_4$ such that $(j-i,k-j)$ is one of $(1,1)$, $(1,2)$, 
 $(2,1)$,
 \begin{eqnarray}
 \label{eq:tet2}
 \frac{qx_{ij}x_{jk} - q^{-1} x_{jk}x_{ij}}{q-q^{-1}}=1.
 \end{eqnarray}
 \item[\rm (iii)] For $i,j,k,\ell \in 
 \Z_4$ such that $j-i=k-j=\ell-k=1$,
 \begin{eqnarray}
 \label{eq:tet3}
 %%%%\label{eq:qserre}
 x^3_{ij}x_{k\ell}
 -
 \lbrack 3 \rbrack_q
 x^2_{ij}x_{k\ell} x_{ij}
 +
 \lbrack 3 \rbrack_q
 x_{ij}x_{k\ell} x^2_{ij}
 -
 x_{k\ell} x^3_{ij} = 0.
 \end{eqnarray}
 \end{enumerate}
 We call
 $\boxtimes_q$ the {\it $q$-tetrahedron algebra}.
\end{definition}

\begin{lemma}
\label{lem:aut1Tet}
There exists an automorphism $\varrho$ of 
$\boxtimes_q$ that sends each generator
$x_{ij} \mapsto x_{i+1,j+1}$.
 Moreover $\varrho^4=1$.
\end{lemma}

\begin{lemma}
\label{lem:mapn}
There exists an $\mathbb F$-algebra homomorphism
$\square_q \to \boxtimes_q$ that sends
$x_i \mapsto x_{i-1,i}$ for $i \in  \mathbb Z_4$.
\end{lemma}
\begin{proof} Compare Definitions
\ref{def:boxqV1M},
\ref{def:tet}.
\end{proof}

\begin{definition}
\label{def:canonBoxTet}
\rm
The homomorphism 
$\square_q \to \boxtimes_q$ from 
Lemma \ref{lem:mapn}
will be called {\it canonical}.
\end{definition}

\noindent Our next goal is to show that the canonical homomorphism
$\square_q \to \boxtimes_q$ is injective.

\begin{definition} 
\label{def:tetsub}
\rm
(See \cite[Section~4.1]{miki}.)
Define the subalgebras 
$ \boxtimes^{\rm even}_q$,
$ \boxtimes^{\rm odd}_q$,
$ \boxtimes^{\times }_q$
of $ \boxtimes_q$ such that
\begin{enumerate}
\item[\rm (i)]
$ \boxtimes^{\rm even}_q$
 is 
generated by $x_{30}, x_{12}$;
\item[\rm (ii)]
$ \boxtimes^{\rm odd}_q$
  is
generated by $x_{01}, x_{23}$;
\item[\rm (iii)]
$ \boxtimes^{\times }_q$
  is
generated by 
$
x_{02}, x_{20},
x_{13}, x_{31}
$.
\end{enumerate}
\end{definition}

\begin{proposition}
\label{prop:miki}
{\rm (See 
Miki \cite[Prop.~4.1]{miki}.)}
The following {\rm (i)--(iv)} hold:
\begin{enumerate}
\item[\rm (i)] 
there exists an $\mathbb F$-algebra isomorphism
$U^+ \to  \boxtimes^{\rm even}_q$ that sends
$X\mapsto x_{30}$ and
$Y\mapsto x_{12}$;
\item[\rm (ii)] 
there exists an $\mathbb F$-algebra isomorphism
$U^+ \to  \boxtimes^{\rm odd}_q$ that sends
$X\mapsto x_{01}$ and
$Y\mapsto x_{23}$;
\item[\rm (iii)] 
the $\mathbb F$-algebra 
$\boxtimes^{\times }_q$ has a presentation by
generators 
$x_{02},
x_{20},
x_{13},
x_{31}$ and relations
\begin{eqnarray*}
x_{02}x_{20} = 
x_{20}x_{02} =  1, \qquad \qquad
x_{13}x_{31} = 
x_{31}x_{13} =  1;
\end{eqnarray*}
\item[\rm (iv)] 
the following is an isomorphism of
$\mathbb F$-vector spaces:
\begin{eqnarray*}
 \boxtimes^{\rm even}_q
\otimes
 \boxtimes^{\times }_q
\otimes 
 \boxtimes^{\rm odd}_q
& \to &  \boxtimes_q
\\
 u \otimes v \otimes w  &\mapsto & uvw
 \end{eqnarray*}
\end{enumerate}
\end{proposition}

\begin{note}
\rm 
In \cite[Prop.~4.1]{miki} Miki
assumes that $\mathbb F = \mathbb C$ and $q$ is
not a root of unity. We emphasize that
Proposition
\ref{prop:miki} holds without this assumption.
There is a proof of
Proposition
\ref{prop:miki} that is
 analogous to our proof of
Proposition
\ref{thm:tensorDecPreM}.
\end{note}

\begin{proposition}
\label{prop:sBI}
The canonical homomorphism
$\square_q \to \boxtimes_q$ is injective.
\end{proposition}
\begin{proof} Compare
Proposition
\ref{thm:tensorDecPreM} and
Proposition
\ref{prop:miki}.
\end{proof}

%%%%%%%%%%%%%%%%%%insert above

\begin{proposition}
\label{cor:tetqmain}
Pick nonzero $a,b \in \mathbb F$.
Then there exists an $\mathbb F$-algebra homomorphism
$\mathcal O \to 
\boxtimes_q$ that sends
\begin{eqnarray*}
A \mapsto a x_{30}+ a^{-1} x_{01},
\qquad \qquad
B \mapsto  b x_{12}+ b^{-1} x_{23}.
\end{eqnarray*}
This homomorphism is injective.
\end{proposition}
\begin{proof}
The desired homomorphism is the composition
of the homomorphism
$\sharp: \mathcal O \to \square_q$ from
Proposition
\ref{prop:ABxyiLongerM}
and the canonical homomorphism 
$\square_q \to \boxtimes_q$.
The last assertion follows from
Theorem
\ref{thm:xiInj}
and  Proposition
\ref{prop:sBI}.
\end{proof}

\section{The quantum affine algebra
$U_q(\widehat {\mathfrak {sl}}_2)$ }

\noindent In the literature on the
$q$-Onsager algebra $\mathcal O$,
there are algebra homomorphisms from $\mathcal O$ into 
the quantum algebra $U_q(\widehat {\mathfrak {sl}}_2)$ due to
P.~Baseilhac and S.~Belliard
 \cite[line (3.15)]{basXXZ}, 
 \cite[line (3.18)]{basXXZ}
 and S.~Kolb
\cite[Example~7.6]{kolb}.
Also, there are algebra homomorphisms from $\mathcal O$ into
the $q$-deformed loop algebra
$U_q(L({\mathfrak{sl}}_2))$ due to
P.~Baseilhac
    \cite[Prop.~2.2]{bas6},
and also
T.~Ito and the present author
\cite[Prop.~8.5]{qRacahIT},
\cite[Props.~1.1,~1.13]{ITaug}.
In the case of \cite[Example~7.6]{kolb} and
\cite[Props.~1.1,~1.13]{ITaug}
the homomorphism was shown to be injective. In fact,
all of the above homomorphisms are injective,
and this can be established using the results from the previous
sections of the present paper.
In each case the proof is similar.
In this section we illustrate what is going on with a single example
\cite[Example~7.6]{kolb}.
This example involves
$U_q(\widehat {\mathfrak {sl}}_2)$, which we now define.

\begin{definition}
\label{def:Chevalley}
\rm 
(See \cite[p.~262]{cp3}.) Let  
$U_q(\widehat {\mathfrak {sl}}_2)$ denote the $\mathbb F$-algebra
with generators $e^{\pm}_i$, $k^{\pm 1}_i$, $i \in \lbrace 0,1\rbrace$
and the following relations:
\begin{eqnarray*}
&& k_i k^{-1}_i = k^{-1}_i k_i = 1,
\\
&& k_0 k_1 = k_1 k_0,
\\
&& k_i e^{\pm}_i k^{-1}_i = q^{\pm 2}e^{\pm}_i,
\\
&& k_i e^{\pm}_j k^{-1}_i = q^{\mp 2}e^{\pm}_j, \qquad i\not=j,
\\
&&
\lbrack e^+_i, e^-_i\rbrack = \frac{k_i - k^{-1}_i}{q-q^{-1}},
\\
&&
\lbrack e^{\pm}_0, e^{\mp}_1\rbrack = 0,
\\
&&
(e^{\pm}_i)^3 e^{\pm}_j
-
\lbrack 3 \rbrack_q 
(e^{\pm}_i)^2 e^{\pm}_j e^{\pm}_i
+
\lbrack 3 \rbrack_q 
e^{\pm}_i e^{\pm}_j (e^{\pm}_i)^2
-
 e^{\pm}_j (e^{\pm}_i)^3 = 0, \qquad i \not=j.
 \end{eqnarray*}
\noindent We call 
$e^{\pm}_i$, $k^{\pm 1}_i$, $i \in \lbrace 0,1\rbrace$
the {\it Chevalley generators} for
$U_q(\widehat {\mathfrak {sl}}_2)$.
\end{definition}

\noindent In the following three lemmas we describe some automorphisms of 
$U_q(\widehat {\mathfrak {sl}}_2)$; the proofs are routine
and omitted.

\begin{lemma}
\label{lem:omegaAut}
There exists an automorphism of
$U_q(\widehat {\mathfrak {sl}}_2)$ that sends
\begin{eqnarray*}
e^+_i \mapsto e^-_i, \qquad \qquad e^-_i \mapsto e^+_i,
\qquad \qquad k^{\pm 1}_i \mapsto k^{\mp 1}_i,
\qquad \qquad i \in \lbrace 0,1\rbrace.
\end{eqnarray*}
\end{lemma}

\begin{lemma} 
\label{lem:autZeta}
There exists
an automorphism  of 
$U_q(\widehat {\mathfrak {sl}}_2)$ that sends
\begin{eqnarray*}
e^+_i \mapsto  e^+_i k_i, 
\qquad \qquad 
e^-_i \mapsto  k^{-1}_i e^-_i,
\qquad \qquad 
k^{\pm 1}_i \mapsto  k^{\pm 1}_i,
\qquad \qquad i \in \lbrace 0,1\rbrace.
\end{eqnarray*}
\end{lemma}

\begin{lemma}
\label{lem:uqsl2Aut}
Let $\varepsilon_0,\varepsilon_1$ denote nonzero scalars in $\mathbb F$.
Then there exists an automorphism of
$U_q(\widehat {\mathfrak {sl}}_2)$ that sends
\begin{eqnarray*}
e^+_i \mapsto \varepsilon_i e^+_i,
\qquad \qquad
e^-_i \mapsto \varepsilon^{-1}_i e^-_i,
\qquad \qquad
k^{\pm 1}_i \mapsto  k^{\pm 1}_i,
\qquad \qquad i \in \lbrace 0,1\rbrace.
\end{eqnarray*}
\end{lemma}

\begin{definition}
\label{def:xi}
\rm
The automorphism of
$U_q(\widehat {\mathfrak {sl}}_2)$ 
from
Lemma \ref{lem:uqsl2Aut}
will be denoted by $\xi(\varepsilon_0,\varepsilon_1)$.
\end{definition}

\noindent We now recall the 
equitable presentation of 
$U_q(\widehat {\mathfrak {sl}}_2)$.

\begin{lemma}
\label{thm:equitUq}
{\rm  (See \cite[Theorem~2.1]{uqsl2hat}.)}
The $\mathbb F$-algebra $U_q(\widehat {\mathfrak {sl}}_2)$ 
has a presentation
 with generators $y^{\pm }_i$, $k^{\pm 1}_i$, $i \in \lbrace 0,1\rbrace$
and relations
\begin{eqnarray*}
&&k_i k^{-1}_i = k^{-1}_i k_i = 1,
\\
&&
\mbox{\rm $k_0 k_1$ is central},
\\
&& \frac{q y^+_i k_i - q^{-1} k_i y^+_i}{q-q^{-1}} = 1,
\\
&& \frac{q k_i y^-_i - q^{-1} y^-_i k_i}{q-q^{-1}} = 1,
\\
&& \frac{q y^{-}_i y^+_i - q^{-1} y^+_i y^{-}_i}{q-q^{-1}} = 1,
\\
&& \frac{q y^{+}_i y^-_j - q^{-1} y^-_j y^{+}_i}{q-q^{-1}} = k^{-1}_0 k^{-1}_1,
\qquad \qquad i \not=j,
\\
&&
(y^{\pm}_i)^3 y^{\pm}_j - 
\lbrack 3 \rbrack_q 
(y^{\pm}_i)^2 y^{\pm}_j y^{\pm}_i 
+
\lbrack 3 \rbrack_q 
y^{\pm}_i y^{\pm}_j (y^{\pm}_i)^2
-
 y^{\pm}_j (y^{\pm}_i)^3 = 0, \qquad \qquad i\not=j.
 \end{eqnarray*}
%%%%%%%%%For nonzero $\xi_0,\xi_1 \in \mathbb F$ 
An isomorphism with the presentation in Definition
\ref{def:Chevalley} is given by:
\begin{eqnarray}
k^{\pm}_i &\mapsto& k^{\pm}_i,
\label{lem:kmove}
\\
y^-_i &\mapsto& k^{-1}_i + e^-_i(q-q^{-1}),
\label{lem:ymmove}
\\
y^+_i &\mapsto & k^{-1}_i - 
k^{-1}_i e^+_i
q(q-q^{-1}).
\label{lem:ypmove}
\end{eqnarray}
\noindent The inverse of this isomorphism is given by:
\begin{eqnarray*}
k^{\pm}_i &\mapsto& k^{\pm}_i,
\\
e^-_i &\mapsto & (y^-_i - k^{-1}_i)(q-q^{-1})^{-1},
\\
e^+_i &\mapsto & (1-k_i y^+_i)q^{-1}(q-q^{-1})^{-1}.
\end{eqnarray*}
 \end{lemma}

\begin{definition}
\label{def:tau}
\rm Referring to Lemma
\ref{thm:equitUq}, we call
$y^{\pm}_i$, $k^{\pm 1}_i$, $i \in \lbrace 0,1\rbrace$
the {\it equitable generators} for 
$U_q(\widehat {\mathfrak {sl}}_2)$.
The isomorphism described in 
(\ref{lem:kmove})--(\ref{lem:ypmove}) will be denoted by
$\tau$.
\end{definition}

\noindent 
We now describe an $\mathbb F$-algebra
homomorphism $\widehat \square_q \to 
U_q(\widehat {\mathfrak {sl}}_2)$ and
an
$\mathbb F$-algebra
homomorphism 
$U_q(\widehat {\mathfrak {sl}}_2) \to
\boxtimes_q$.
 In this description we use
the equitable presentation of 
$U_q(\widehat {\mathfrak {sl}}_2)$. 

\begin{lemma}
\label{lem:sigmai}
There exists an $\mathbb F$-algebra
homomorphism $\widehat \sigma: \widehat \square_q \to 
U_q(\widehat {\mathfrak {sl}}_2)$ such that
\bigskip

\centerline{
\begin{tabular}[t]{c||cccc|cccc}
$u$
& $x_0$ & $x_1$ & $x_2$ 
& $x_3$ &
$c_0$ & $c_1$ & $c_2$ & $c_3$
\\
\hline
$\widehat \sigma(u)$ &
$y^-_0$ & $y^+_0$ & $y^-_1$ & $y^+_1$ &
$1$ & $k^{-1}_0 k^{-1}_1$ & $1$ & $k^{-1}_0 k^{-1}_1$
     \end{tabular}}
\medskip

\end{lemma}
\begin{proof} Compare the defining relations
(\ref{eq:Box1})--(\ref{eq:Box4})
for $\widehat \square_q$, with the defining relations
for 
$U_q(\widehat {\mathfrak {sl}}_2)$ given in Lemma
\ref{thm:equitUq}.
\end{proof}

\begin{lemma}
\label{lem:sigmaiT}
There exists an $\mathbb F$-algebra
homomorphism $ \sigma: 
U_q(\widehat {\mathfrak {sl}}_2) \to
\boxtimes_q$ 
such that 
\bigskip

\centerline{
\begin{tabular}[t]{c||cccc|cccc}
$u$ & $y^-_0$ & $y^+_0$ & $y^-_1$ 
& $y^+_1$ & $k_0$ & $k^{-1}_0$ & $k_1$ & $k^{-1}_1$
   \\ \hline  
$\sigma(u)$ 
& $x_{30}$ & $x_{01}$ & $x_{12}$ & $x_{23}$
& $x_{13}$ & $x_{31}$ & $x_{31}$ & $x_{13}$
\\
     \end{tabular}}
\medskip

\end{lemma}
\begin{proof}
Compare the defining relations for
$U_q(\widehat {\mathfrak {sl}}_2)$ given in
Lemma 
\ref{thm:equitUq},
with the defining relations for
$\boxtimes_q$ given in
Definition
\ref{def:tet}.
\end{proof}

\begin{lemma}
\label{lem:backtoCan}
The following diagram commutes:

\begin{equation*}
\begin{CD}
\widehat \square_q @>\widehat \sigma > >
                      U_q(\widehat {\mathfrak {sl}}_2) 
	   \\ 
          @V can VV                   @VV \sigma V \\
            \square_q @>>can > 
             \boxtimes_q 
                   \end{CD}
\end{equation*}

\end{lemma}
\begin{proof} 
Each map in the diagram is an $\mathbb F$-algebra homomorphism.
To verify that the diagram commutes, chase the 
$\widehat \square_q $
generators $x_i, c^{\pm 1}_i$ $(i \in \mathbb Z_4)$ around the diagram using
Definitions
\ref{def:cancan},
\ref{def:canonBoxTet}
and
Lemmas \ref{lem:sigmai},
\ref{lem:sigmaiT}.
\end{proof}

%%%%%%%%%%%%%%%%%%%%%%%%%%%%%%%%%%%%%%%%%%%%%%%%%%%%%%%

\noindent We now obtain
the homomorphism $\mathcal O \to
U_q(\widehat {\mathfrak {sl}}_2)$ due to  Kolb
 \cite[Example~7.6]{kolb}.

\begin{proposition}
\label{prop:Kolb}
{\rm (See \cite[Example~7.6]{kolb}.)}
For
$i \in \lbrace 0,1\rbrace$ pick
$s_i\in \mathbb F$ and $0 \not= \gamma_i \in \mathbb F$.
Then 
for $ U_q(\widehat {\mathfrak {sl}}_2)$
the elements
\begin{eqnarray}
B_i = e^-_i -\gamma_i e^+_i k^{-1}_i + s_i k^{-1}_i
\qquad \qquad i \in \lbrace 0,1\rbrace
\label{eq:BiKolb}
\end{eqnarray}
satisfy
\begin{eqnarray}
&&
B_0^3 B_1 - \lbrack 3 \rbrack_q B_0^2 B_1 B_0 +
\lbrack 3 \rbrack_q B_0 B_1 B_0^2 -B_1 B_0^3 = q(q+q^{-1})^2 \gamma_0 
(B_1B_0-B_0B_1),
\label{eq:AAABKolb}
\\
&&
B_1^3 B_0 - \lbrack 3 \rbrack_q B_1^2 B_0 B_1 +
\lbrack 3 \rbrack_q B_1 B_0 B_1^2 -B_0 B_1^3 = 
q(q+q^{-1})^2 \gamma_1(B_0B_1-B_1B_0).
\label{eq:BBBAKolb}
\end{eqnarray}
\end{proposition}
\begin{proof}
Let $\lambda$ denote an indeterminate.
Replacing $\mathbb F$ by its algebraic closure if necessary, we 
may assume without loss that 
$\mathbb F$ is algebraically closed.
There exist scalars 
$\lbrace \alpha_i \rbrace_{i \in \mathbb Z_4}$  in $\mathbb F$
such that
$\alpha_0, \alpha_1$ are the roots of the
polynomial
\begin{eqnarray*}
\lambda^2 - s_0 \lambda + \gamma_0 q (q-q^{-1})^{-2}
\end{eqnarray*}
and $\alpha_2, \alpha_3$ are the roots of the polynomial 
\begin{eqnarray*}
\lambda^2 - s_1 \lambda + \gamma_1 q (q-q^{-1})^{-2}.
\end{eqnarray*}
By construction
\begin{eqnarray*}
&&
\alpha_0 + \alpha_1 
= 
s_0,
\qquad \qquad 
\alpha_0 \alpha_1 = \gamma_0 q (q-q^{-1})^{-2},
\\
&&
\alpha_2 + \alpha_3 
= 
s_1,
\qquad \qquad 
\alpha_2 \alpha_3 = \gamma_1 q (q-q^{-1})^{-2}.
\end{eqnarray*}
Note that $\alpha_i \not=0$ for $i \in \mathbb Z_4$.
Define
\begin{eqnarray*}
\varepsilon_0 = \alpha_0 (q-q^{-1}),
\qquad \qquad 
\varepsilon_1 = \alpha_2 (q-q^{-1}).
\end{eqnarray*} 
Note that $\varepsilon_0 \not=0$ and $\varepsilon_1 \not=0$.
In the algebra $\widehat \square_q$ define
\begin{eqnarray*}
A = \alpha_0 x_0 + \alpha_1 x_1, \qquad \qquad
B = \alpha_2 x_2 + \alpha_3 x_3.
\end{eqnarray*}
Then
$A,B$ satisfy
(\ref{eq:qDGLong1}), 
(\ref{eq:qDGLong2}) by
Corollary
\ref{prop:ABxyiLong}.
 Consider the composition
\begin{eqnarray}
\label{eq:cdlong}
\begin{CD} 
\zeta: \quad \widehat \square_q  @>>  \widehat \sigma >  
U_q(\widehat {\mathfrak {sl}}_2)
@>> \tau > 
U_q(\widehat {\mathfrak {sl}}_2)
               @>>
	      \xi(\varepsilon_0,\varepsilon_1) 
	       > 
U_q(\widehat {\mathfrak {sl}}_2),
		  \end{CD}
\end{eqnarray}
where $\widehat \sigma$ is from
Lemma
\ref{lem:sigmai}, 
$\tau$ is from
Definition
\ref{def:tau},
and 
$\xi(\varepsilon_0,\varepsilon_1)$ is from
Definition \ref{def:xi}.
By construction $\zeta:\widehat \square_q \to
U_q(\widehat {\mathfrak {sl}}_2)$ is an $\mathbb F$-algebra homomorphism.
Using (\ref{eq:cdlong}) and
 $k_0 e^+_0 = q^{2} e^+_0 k_0$ we find that
$\zeta$ sends $A\mapsto B_0$. Similarly $\zeta$ sends
 $B \mapsto B_1$.
By 
Lemma
\ref{lem:sigmai}, 
$\widehat \sigma$ sends $c_0\mapsto 1$ and $c_2 \mapsto 1$. Therefore
$\zeta$ sends
$c_0\mapsto 1$ and $c_2 \mapsto 1$.
Applying $\zeta$ to each side of
(\ref{eq:qDGLong1}), 
(\ref{eq:qDGLong2})
we  obtain
(\ref{eq:AAABKolb}),
(\ref{eq:BBBAKolb}).
\end{proof}

\begin{proposition}
\label{prop:KolbB}
{\rm (See \cite[Example~7.6]{kolb}.)}
Referring to Proposition 
\ref{prop:Kolb}, assume that
\begin{eqnarray}
\label{eq:gammaChoice}
\gamma_0 = q^{-1}(q-q^{-1})^2, \qquad \qquad
\gamma_1 = q^{-1}(q-q^{-1})^2.
\end{eqnarray}
Then there exists an $\mathbb F$-algebra homomorphism
$\mathcal O \to 
U_q(\widehat {\mathfrak {sl}}_2)$ that sends
$A\mapsto B_0$ and
$B\mapsto B_1$.
This homomorphism is injective.
\end{proposition}
\begin{proof} The desired $\mathbb F$-algebra
homomorphism
$\mathcal O \to U_q(\widehat {\mathfrak {sl}}_2)$ 
exists, since under the assumption 
(\ref{eq:gammaChoice}) the relations
(\ref{eq:AAABKolb}),
(\ref{eq:BBBAKolb})
become the $q$-Dolan/Grady relations.
Call the above homomorphism $\partial$.
We show that $\partial$ is injective.
Consider the composition
\begin{eqnarray*}
\begin{CD} 
\eta: 
\quad
\mathcal O @>> \partial >  
U_q(\widehat {\mathfrak {sl}}_2)
@>> \xi(\varepsilon^{-1}_0, \varepsilon^{-1}_1) > 
U_q(\widehat {\mathfrak {sl}}_2)
               @>>
	      \tau^{-1} 
	       > 
U_q(\widehat {\mathfrak {sl}}_2)
@>> \sigma >  \boxtimes_q,
		  \end{CD}
\end{eqnarray*}
where 
$\xi(\varepsilon^{-1}_0,\varepsilon^{-1}_1)$ is from
Definition \ref{def:xi},
$\tau$ is from
Definition
\ref{def:tau}, and
$\sigma$ is from
Lemma
\ref{lem:sigmaiT}.
By construction
$\eta:\mathcal O
\to
\boxtimes_q$
is an $\mathbb F$-algebra homomorphism.
One checks that $\eta$
coincides with the $\mathbb F$-algebra homomorphism
$\mathcal O \to \boxtimes_q$ from Proposition
\ref{cor:tetqmain}, where $a=\alpha_0$ and $b=\alpha_2$.
The map from Proposition
\ref{cor:tetqmain} is injective, 
so $\partial$ is injective.
\end{proof}

\begin{note}\rm
In 
\cite[Example~7.6]{kolb} Kolb assumes
 that $\mathbb F$ has
characteristic zero and $q$ is not a root of unity.
We emphasize that
Propositions
\ref{prop:Kolb} and
\ref{prop:KolbB} hold without this assumption.
\end{note}

%%%%%%%%%%%%%%%%%%%%%%%%%%%%%%%%%%%%

%%%%%%%%%%%%%%%%%%%%%%%%%%%%%%%%%%%%%%%%%%%%%%%%%%%%%%%%%

\section{Directions for future research}
\noindent In this section we give some suggestions for future research.

\begin{problem}\rm
For the $\mathbb F$-algebras 
$\widetilde \square_q$,
$\widehat \square_q$, $\square_q$ find their
automorphism group.
\end{problem}

\begin{problem}\rm
The Lusztig automorphisms of  $U_q(\widehat {\mathfrak {sl}}_2)$ are
described in
\cite[p.~294]{damiani}. Find analogous automorphisms for
$\widetilde \square_q$,
$\widehat \square_q$, $\square_q$.
\end{problem}

\begin{problem}\rm
\label{prob:N}
Referring to  the algebra $\square_q$, for
$i \in \mathbb Z_4$ define
\begin{eqnarray*}
N_i = \frac{q(1-x_ix_{i+1})}{q-q^{-1}} =
\frac{q^{-1}(1-x_{i+1}x_i)}{q-q^{-1}}.
\end{eqnarray*}
Note that $N_i x_i = q^{2}x_i N_i$ and
$N_i x_{i+1} = q^{-2}x_{i+1} N_i$.
Determine how $N_i$ is related to $x_{i+2}$ and $x_{i+3}$.
\end{problem}

\begin{problem}\rm
Referring to Problem \ref{prob:N},
determine how $N_i$, $N_j$ are
related for $i,j \in \mathbb Z_4$.
\end{problem}

\begin{problem}\rm
Referring to Problem \ref{prob:N},
consider the $q$-exponential
$E_i={\rm exp}_q (N_i)$.   
For $u \in \square_q$ compute
$E_i u E_i^{-1}$ and determine if the result is contained in
$\square_q$. If it always is, then conjugation by $E_i$ gives an
automorphism of $\square_q$. In this case, describe the subgroup of
${\rm Aut}(\square_q)$ generated by
$\lbrace E^{\pm 1}_i\rbrace_{i \in \mathbb Z_4}$.
\end{problem}

\begin{conjecture}\rm
Let $V$ denote a finite-dimensional irreducible
 $\widetilde \square_q$-module. Then 
$V$ becomes a $\boxtimes_q$-module
such that for $i \in \mathbb Z_4$
the action of $x_i$ on $V$ is a scalar multiple of 
the action of $x_{i-1,i}$ on $V$. The
$\boxtimes_q$-module $V$ is  irreducible.
\end{conjecture}

\begin{conjecture}\rm Let $A,B$ denote a tridiagonal pair over $\mathbb F$
that has $q$-Racah type in the sense of
\cite[p.~259]{bockting}.
Then the underlying vector space $V$ becomes a $\square_q$-module
on which $A$ (resp. $B$) is a linear combination of $x_0, x_1$ 
(resp. 
$x_2, x_3$).
The $\square_q$-module $V$ is irreducible.
\end{conjecture}

\begin{problem}\rm
For the $\mathbb F$-algebras 
$\widetilde \square_q$,
$\widehat \square_q$, $\square_q$ find their center.
\end{problem}

%%%%%%%%%%%
%\begin{problem} \rm
%Similar to the previous problem, but work
%with the tetrahedron algebra instead of the
%$q$-tetrahedron algebra.
%\end{problem}
%%%%%%%%%%%

\section{Acknowledgments}
The author thanks 
 %Sarah Bockting-Conrad,
  %Aiping Deng,
   %Jae-ho Lee,
    %Alison Gordon Lynch,
     %Gabriel Pretel,
      %Edward Hanson,
       %Thao Vu
      Pascal Baseilhac, Stefan Kolb,
       and Kazumasa Nomura 
      for giving this paper a close reading and offering valuable
      suggestions.

\noindent Paul Terwilliger \hfil\break
\noindent Department of Mathematics \hfil\break
\noindent University of Wisconsin \hfil\break
\noindent 480 Lincoln Drive \hfil\break
\noindent Madison, WI 53706-1388 USA \hfil\break
\noindent email: {\tt terwilli@math.wisc.edu }\hfil\break


\begin{thebibliography}{10}
%%%%%%%%%


\bibitem{bas2}
P.~Baseilhac.
\newblock An integrable structure related with tridiagonal
algebras.
\newblock {\em
Nuclear Phys. B}
 705
(2005)
605--619;
{\tt arXiv:math-ph/0408025}.

\bibitem{bas1}
P.~Baseilhac.
\newblock Deformed {D}olan-{G}rady relations in quantum integrable
models.
\newblock {\em
Nuclear Phys. B}
 709
(2005)
491--521;
{\tt arXiv:hep-th/0404149}.
 

\bibitem{BK05}
P.~Baseilhac and K.~Koizumi.
\newblock A new (in)finite dimensional algebra for
quantum integrable models.
\newblock {\em
 Nuclear Phys. B}  720  (2005) 325--347;
  {\tt arXiv:math-ph/0503036}.

  \bibitem{bas4}
 P.~Baseilhac and K.~Koizumi.
  \newblock
 A deformed analogue of Onsager's symmetry
  in the
 $XXZ$ open spin chain.
\newblock {\em
 J. Stat. Mech. Theory Exp.}  2005,  no. 10, P10005, 15 pp. (electronic);
  {\tt arXiv:hep-th/0507053}.

\bibitem{bas5}
 P.~Baseilhac.
 \newblock
 The $q$-deformed analogue of the Onsager algebra:
  beyond the Bethe ansatz approach.
 \newblock {\em Nuclear Phys. B}  754
 (2006) 309--328;
 {\tt arXiv:math-ph/0604036}.


    \bibitem{bas6}
    P.~Baseilhac.
    \newblock
    A family of tridiagonal pairs and related symmetric functions.
    \newblock
    {\em J. Phys. A}  39
    (2006) 11773--11791;
    {\tt arXiv:math-ph/0604035}.

   \bibitem{bas7}
    P.~Baseilhac and K.~Koizumi.
    \newblock
     Exact spectrum of the $XXZ$ open spin chain from the
      $q$-Onsager algebra representation theory.
        \newblock {\em  J. Stat. Mech. Theory Exp.}
     2007,  no. 9, P09006, 27 pp. (electronic);
       {\tt arXiv:hep-th/0703106}.

 \bibitem{bas8}
 P.~Baseilhac and S.~Belliard.
 \newblock
 Generalized $q$-Onsager algebras and boundary affine Toda field theories.
 \newblock {\em
 Lett. Math. Phys.} {\bf 93}  (2010)  213--228.


 \bibitem{basnc}
 P.~Baseilhac and K.~Shigechi.
 \newblock
 A new current algebra and the reflection equation.
 \newblock{\em
 Lett. Math. Phys. }
  92
 (2010)   47--65.

\bibitem{basXXZ}
P.~Baseilhac, S.~Belliard.
\newblock
The half-infinite XXZ chain in Onsager's approach.
\newblock
{\em 
Nuclear Phys. B} 873 (2013) 550--584. 

\bibitem{basKojima}
P.~Baseilhac, T.~Kojima.
\newblock
Correlation functions of the half-infinite XXZ 
spin chain with a triangular boundary. 
\newblock{\em
Nuclear Phys. B} 880 (2014) 378--413. 

%\bibitem{bvu}
%P.~Baseilhac, T.T.~Vu.
%\newblock
%Analogues of Lusztig's higher order relations for the
%$q$-Onsager algebra.
%{\tt arXiv:1312.3433}

\bibitem{beck3}
J.~Beck.
\newblock
Braid group action and quantum affine algebras.
\newblock {\em
Commun. Math. Phys.} 
165 (1994) 555--568.


\bibitem{beck2}
J.~Beck, V.~Chari, A.~Pressley.
\newblock
An algebraic characterization of the affine canonical basis.
\newblock
{\em
Duke Math. J.} 99 (1999) 455--487.


\bibitem{beck}
J.~Beck, H.~Nakajima.
\newblock
Crystal bases and two-sided cells of quantum affine
algebras.
\newblock
{\em Duke Math. J.} 123 (2004) 335--402.

\bibitem{bc}
S.~Belliard, N.~Crampe.
\newblock 
Coideal algebras from twisted Manin triples.
\newblock
{\em J. Geom. Phys.} 62 (2012) 2009--2023;
{\tt arXiv:1202.2312}

%\bibitem{bocktingTer}
%S.~Bockting-Conrad, P.~Terwilliger.
%\newblock
%The algebra 
%$U_q({\mathfrak {sl}}_2)$ in disguise.
%\newblock {\em
%Linear Algebra Appl.} 459 (2014) 548--585.

\bibitem{bockting}
S.~Bockting-Conrad.
\newblock
Tridiagonal pairs of $q$-Racah type,
the double lowering operator $\psi$, and the quantum algebra 
$U_q(\mathfrak{sl}_2)$.
\newblock {\em
Linear Algebra Appl.} 445 (2014) 256--279.


\bibitem{carter}
R.~Carter.
\newblock  
{\em Lie algebras of finite and affine type}.
\newblock Cambridge Studies in Advanced Mathematics 96.
Cambridge University Press, 2005.


\bibitem{cp3}
V.~Chari and A.~Pressley.
\newblock Quantum affine algebras.
\newblock{\em 
Commun. Math. Phys.}
142 (1991) 261--283.

%%%%%%%%%%%%%%%%%%
%\bibitem{hI}
%T.~Hattai, T.~Ito.
%\newblock
%On a certain subalgebra of Uq(sl^2) related to the 
%degenerate q-Onsager algebra.
%\newblock{\em
%SIGMA Symmetry Integrability Geom. Methods Appl.} 11 (2015), Paper 007,
%13 pp.
%%%%%%%%%%%%%%%%
\bibitem{damiani}
I.~Damiani.
\newblock
A basis of type Poincare-Birkoff-Witt for the quantum
algebra of $\widehat{\mathfrak{sl}}_2$.
\newblock {\em
J. Algebra} 161 (1993) 291--310.



%\bibitem{TDclass}
%T.~Ito, K.~Nomura,  P.~Terwilliger.
%\newblock
%A classification of sharp tridiagonal pairs.
%\newblock{\em Linear Algebra Appl.} {\bf 435} (2011) 1857--1884.
%{\tt arXiv:1001.1812}.



\bibitem{irt}
T.~Ito, H.~Rosengren, P.~Terwilliger. 
\newblock
Evaluation modules for the $q$-tetrahedron algebra.
\newblock{\em
Linear Algebra Appl.}  451 (2014) 107--168;
\newblock
{\tt arXiv:1308.3480}.

\bibitem{TD00}
T.~Ito, K.~Tanabe, P.~Terwilliger.
\newblock Some algebra related to ${P}$- and ${Q}$-polynomial association
 schemes,  in:
 \newblock {\em Codes and Association Schemes (Piscataway NJ, 1999)}, Amer.
  Math. Soc., Providence RI, 2001, pp.
        167--192;
    {\tt arXiv:math.CO/0406556}.





\bibitem{shape}
T.~Ito, P.~Terwilliger. 
\newblock
The shape of a tridiagonal pair.
\newblock{\em
Journal of Pure and Applied Algebra} 188 (2004) 145--160.


\bibitem{uqsl2hat}
T.~Ito and P.~Terwilliger.
\newblock
Tridiagonal pairs and the quantum affine algebra
$U_q(\widehat {\mathfrak {sl}}_2)$. 
\newblock
{\em
Ramanujan J.} 13 (2007) 39--62;
\newblock
{\tt arXiv:math/0310042}.

\bibitem{nonnil}
T.~Ito and P.~Terwilliger.
\newblock
Two non-nilpotent linear transformations that satisfy the cubic $q$-Serre 
relations.
\newblock {\em
J. Algebra Appl.}  6 (2007) 477--503;
\newblock
{\tt arXiv:math/0508398}.

\bibitem{qtet}
T.~Ito and P.~Terwilliger.
\newblock
The $q$-tetrahedron algebra and its finite-dimensional irreducible modules.
\newblock
{\em Comm. Algebra}  35 (2007) 3415--3439;
\newblock
{\tt arXiv:math/0602199}.

\bibitem{qRacahIT}
T.~Ito and P.~Terwilliger.
\newblock
Tridiagonal pairs of $q$-Racah type.
\newblock{\em 
J. Algebra} 322 (2009) 68--93.

\bibitem{ITaug}
T.~Ito and P.~Terwilliger.
\newblock
The augmented tridiagonal algebra.
\newblock{\em 
Kyushu J. Math.} 64 (2010) 81--144. 


%\bibitem{equit}
%T.~Ito, P.~Terwilliger, C.~Weng.
%\newblock
%The quantum algebra $U_q(\mathfrak{sl}_2)$ and its equitable presentation.
%\newblock{\em
%J. Algebra}  298  (2006) 284--301.
%\newblock
%{\tt arXiv:math/0507477}. 
%%%%%%%%%%%%%

%\bibitem{jantzen}
%J.~Jantzen.
%\newblock
%{\em Lectures on quantum groups}.
%\newblock 
%Graduate Studies in Mathematics, 6. Amer. Math. Soc., 
%Providence, RI, 1996. 

\bibitem{kolb}
S.~Kolb.
\newblock
Quantum symmetric Kac-Moody pairs.
\newblock {\em
Adv. Math.} 267 (2014) 395-469.

\bibitem{lusztig}
G.~Lusztig.
\newblock
{\em
Introduction to Quantum Groups}.
\newblock
Progress in Mathematics, Vol. 110, Birkhauser, Boston, 1993.




\bibitem{miki}
K.~Miki.
\newblock
Finite dimensional modules for the $q$-tetrahedron algebra.
\newblock {\em
Osaka J. Math}. 47 (2010) 559--589. 

\bibitem{tersub3}
P.~Terwilliger.
\newblock The subconstituent algebra of an association scheme III.
\newblock{\em
J. Algebraic Combin. }
2  (1993) 177--210.

\bibitem{LS99}
P.~Terwilliger.
\newblock Two linear transformations each tridiagonal with respect to an
 eigenbasis of the other.
  \newblock {\em Linear Algebra Appl.}   330 (2001) 149--203;
    {\tt arXiv:math.RA/0406555}.




\bibitem{TwoRel}
P.~Terwilliger.
\newblock
Two relations that generalize the $q$-Serre relations and the
Dolan-Grady relations.
\newblock{\em
Physics and combinatorics 1999 (Nagoya), 377-398},
\newblock
World Sci. Publ., River Edge, NJ, 2001.




%\bibitem{tersym}
%P.~Terwilliger.
%\newblock
%The equitable presentation for the quantum group 
%$U_q(\mathfrak{g})$ associated with a symmetrizable Kac-Moody algebra
%$\mathfrak{g}$.
%\newblock{\em
%J. Algebra} 298 (2006) 302--319.
%\newblock
%{\tt arXiv:math/0507478}.
%%%%%%%%%%%%%%%%%

\bibitem{uaw}
P.~Terwilliger.
\newblock
The universal Askey-Wilson algebra. 
\newblock
{\em
SIGMA Symmetry Integrability Geom. Methods Appl.}
7 (2011), Paper 069, 24 pp.


%\bibitem{uawe}
%P. Terwilliger.
%\newblock
%The universal Askey-Wilson algebra and the equitable
%presentation of
%$U_q(\mathfrak{sl}_2)$.
%\newblock {\em 
%SIGMA} 
%{\bf 7} (2011) 099, 26 pp.
%{\tt arXiv:1107.3544}.
%%%%%%%%%%%

%\bibitem{fduq}
%P.~Terwilliger.
%\newblock
%Finite-dimensional irreducible
%$U_q(\mathfrak{sl}_2)$-modules from the
%equitable point of view.
%\newblock{\em Linear Algebra Appl.}  439 (2013) 358--400.
%\newblock {\tt arXiv:1303.6134}.
%%%%%%%%%%%%%%%%%%



 \end{thebibliography}
\end{document}